\documentclass[a4paper,12pt]{article}

\usepackage[utf8]{inputenc}

  \usepackage[T1]{fontenc}
  \usepackage[utf8]{inputenc}
  \usepackage{lmodern}
  \usepackage[top=2cm,bottom=3cm, left=2.5cm, right=2.5cm]{geometry}
  \usepackage{amsmath, amssymb}
  \usepackage{hyperref}
  \usepackage{color}
  \usepackage{graphicx}
  \usepackage{tikz}
  \usepackage{tikzrput}
  \usepackage{verbatim}
  \usepackage{mathrsfs}
  \usepackage{cancel}
  \usepackage[affil-it]{authblk}
 
  \usepackage{amsthm}
  \usepackage{enumerate}

  \theoremstyle{plain}
  \newtheorem{thm}{Theorem}[section]
  \newtheorem*{theorem-non}{Theorem}
  \newcounter{thI}
  \newtheorem{thmInt}[thI]{Theorem}

  \newtheorem{lem}{Lemma}[section]
  \newtheorem{prop}{Proposition}[section]
  \newtheorem{cor}{Corollary}[section]

  \newcounter{cntasmp}

  \newtheorem{assumption}[cntasmp]{Assumption}

  \theoremstyle{definition}

  \theoremstyle{remark}
  \newtheorem{rmq}{Remark}[section]
  
\def\ptn(#1)(#2){\fill(#1)circle(2pt)}

\begin{document}

  \renewcommand{\proofname}{Proof}
  \renewcommand\thethI{\arabic{thI}} 
  
  \author{Laurent Dietrich%
  \thanks{Electronic address: \texttt{laurent.dietrich@math.univ-toulouse.fr}}}
\affil{Institut de Mathématiques de Toulouse ; UMR5219 \\ Université de Toulouse ; CNRS \\ UPS IMT, F-31062 Toulouse Cedex 9, France}
  
  \title{Velocity enhancement of reaction-diffusion fronts by a line of fast diffusion}
  \maketitle
  
  \begin{abstract}
  We study the velocity of travelling waves of a reaction-diffusion system coupling a standard reaction-diffusion equation in a strip with a one-dimensional diffusion equation on a line. We show that it grows like the square root of the diffusivity on the line. This generalises a result of Berestycki, Roquejoffre and Rossi in the context of Fisher-KPP propagation where the question could be reduced to algebraic computations. Thus, our work shows that this phenomenon is a robust one. The ratio between the asymptotic velocity and the square root of the diffusivity on the line is characterised as the unique admissible velocity for fronts of an hypoelliptic system, which is shown to admit a travelling wave profile.    
  \end{abstract}

\newpage

\section{Introduction}

This paper deals with the limit $D \to +\infty$ of the following system with unknowns $c~>~0, u(x), v(x,y)$~:
\begin{alignat*}{1}
&\begin{cases}
-d\Delta v+ c\partial_x v = f(v)\text{ for }(x,y)\in \Omega_L := \mathbb R\times]-L,0[ \\
d\partial_y v(x,0) = \mu u(x) - v(x,0)\text{ for }x\in \mathbb R \\
-d\partial_y v(x,-L) = 0 \text{ for }x\in \mathbb R\\ 
-Du''(x) + cu'(x) = v(x,0) - \mu u(x)\text{ for }x\in \mathbb R
\end{cases}
\end{alignat*}
along with the uniform in $y$ limiting conditions
$$\mu u, v \to 0 \text{ as }x\to -\infty$$
$$\mu u, v \to 1 \text{ as } x\to +\infty$$
These equations will be represented from now on as the following diagram 
\begin{equation}
\label{normal}
  \begin{tikzpicture}
  \draw (-6,0) -- (6,0) node[pos=0.5,below] {\small{$d\partial_y v = \mu u - v$}} node[pos=0.5,above] {$-Du'' + cu' = v - \mu u$};

  \node at (4.3,0.33) {$u\to 1/\mu$};
  \node at (-4.3,0.33) {$0 \leftarrow u$};

  \node at (0,-1.5) {$- d\Delta v + c\partial_x v = f(v)$};

  \draw (-6,-3) -- (6,-3) node[pos=0.5,above] {\small{$-\partial_y v= 0$}};


  \node at (4.3,-1.5) {$v \to 1$};
  \node at (-4.3,-1.5) {$0 \leftarrow v$};

  \end{tikzpicture}
\end{equation}

In \cite{BRR}, Berestycki, Roquejoffre and Rossi introduced the following reaction-diffusion system :
\begin{equation}
\label{readi}
  \begin{tikzpicture}
  \draw (-6,0) -- (6,0) node[pos=0.5,below] {\small{$d\partial_y v = \mu u - v$}} node[pos=0.5,above] {$\partial_t u - D\partial_{xx}u  = v - \mu u$};

  \node at (0,-1.5) {$\partial_t v - d\Delta v  = f(v)$};

  \draw (-6,-3) -- (6,-3) node[pos=0.5,above] {\small{$-\partial_y v= 0$}};


  \end{tikzpicture}
\end{equation}
but in the half plane $y < 0$ with $f(v)$ of the KPP-type, i.e $f > 0$ on $(0,1)$, $f(0) = f(1) = 0$, $f'(1) < 0$ and $f(v) \leq f'(0)v$. Such a system was proposed to give a mathematical description of the influence of transportation networks on biological invasions. If $(c,u,v)$ is a solution of \eqref{normal}, then $(u(x+ct),v(x+ct))$ is a travelling wave solution of \eqref{readi}, connecting the states $(0,0)$ and $(1/\mu,1)$. In \cite{BRR}, the following was shown :
\begin{theorem-non}(\cite{BRR})
\begin{enumerate}[i)]
\item Spreading. There is an asymptotic speed of spreading $c_* = c_*(\mu,d,D) > 0$ such that the following is true. Let the initial datum $(u_0,v_0)$ be compactly supported, non-negative and $\not\equiv (0,0)$. Then :
\begin{itemize}
\item for all $c> c_*$ $$\lim_{t\to+\infty} \sup_{|x|\geq ct} (u(x,t), v(x,y,t)) = (0,0)$$ uniformly in $y$.
\item for all $c < c_*$ $$\lim_{t\to+\infty} \inf_{|x|\leq ct} (u(x,t), v(x,y,t)) = (1/\mu,1)$$ locally uniformly in $y$.
\end{itemize}
\item  The spreading velocity. If $d$ and $\mu$ are fixed the following holds true.
\begin{itemize}
\item If $D \leq 2d$, then $c_*(\mu,d,D) = c_{KPP} =: 2\sqrt{df'(0)}$
\item If $D > 2d$ then $c_*(\mu,d,D) > c_{KPP}$ and $\lim_{D\to+\infty} c_*(\mu,d,D)/\sqrt{D}$ exists and is a positive real number.
\end{itemize}
\end{enumerate}
\end{theorem-non}
\ \\
\indent
Thus a relevant question is whether the result of \cite{BRR} is due to the particular structure of the nonlinearity or if it has a more universal character. This is a non trivial question since the KPP case benefits from the very specific property $f(v) \leq f'(0)v$ : in such a case propagation is dictated by the linearised equation near $0$, and the above question can be reduced to algebraic computations. Observe also that some enhancement phenomena really need this property : for instance, for the fractional reaction-diffusion equation
$$\partial_t u + (-\Delta)^s u = u(1-u)$$
in \cite{CR13, CCR}, Cabré, Coulon and Roquejoffre proved that the propagation of an initially compactly supported datum is exponential in time. Nonetheless, this property becomes false and propagation stays linear in time with the reaction term studied here, as proved by Mellet, Roquejoffre and Sire in \cite{MRS}. In this paper, we will show that the phenomenon highlighted in \cite{BRR} persists under a biologically relevant class of nonlinearities that arise in the modelling of Allee effect. Namely $f$ will be of the ignition type :
\begin{assumption}
\label{asf}
 $f:[0,1] \to \mathbb R $ is a smooth non-negative function, $f = 0$ on $[0,\theta]\cup\{1\}$ with $\theta > 0$, $f(0) = f(1) = 0$, and $f'(1) < 0$. For convenience we will still call $f$ an extension of $f$ on $\mathbb R$ by zero at the left of $0$ and by its tangent at $1$ (so it is negative) at the right of $1$. 
\end{assumption}
\begin{figure}[h!]
 \centering \def\svgwidth{200pt} 
 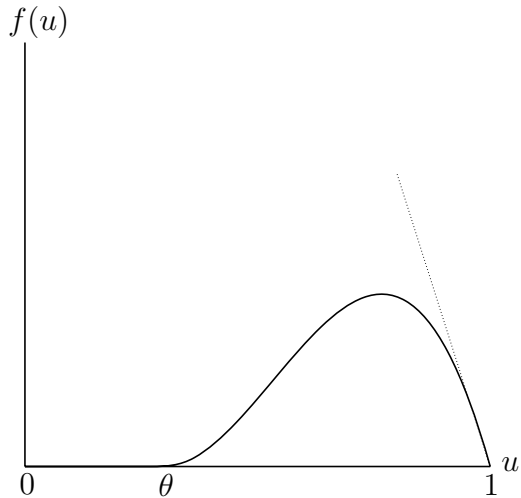 \caption{Example $f = \mathbf 1_{u>\theta}(u-\theta)^2(1-u)$} 
\end{figure}

With our choice of $f$, dynamics in the system \eqref{readi} is governed by the travelling waves, which explains our point of view to answer the question through the study of equation \eqref{normal}. Replacing the half-plane of \cite{BRR} by a strip is a technical simplification, legitimate since we are only interested in the propagation in the direction $x$. Observe that in the light of \cite{BRR3} and the numerical simulations in \cite{Mathese}, translating our results in the half-plane setting seems to be a deep and non-trivial question that goes outside the scope of this paper and will be studied elsewhere.

Our starting point is the following result :
\begin{thmInt}(\cite{LD14})

Let $f$ satisfy Assumption \eqref{asf}. Then there exists $c(D) >0$ and $u, v$ smooth solutions of \eqref{normal}. Moreover, $c(D)$ is unique, $u$ and $v$ are unique up to translations in the $x$ direction, and $u', \partial_x v > 0$.
\end{thmInt}
The first result we will prove is the following :
\begin{thmInt}
\label{cD}
 There exists $c_\infty > 0$ such that $$c(D) \sim_{D\to +\infty} c_\infty \sqrt{D}$$
\end{thmInt}
\begin{rmq}
\label{rmqhomo}
 We would like to point out that in the homogeneous equation in $\mathbb R^n$  : 
\begin{equation}
\label{rn}
-d\Delta v + c\partial_x v = f(v)
\end{equation} it is trivial by uniqueness (see the works of Kanel \cite{KA}) that $c(d) = c_0 \sqrt{d}$ where $c_0$ is the velocity solution of \eqref{rn} with $d=1$. Indeed, to see this, just rescale \eqref{rn} by $\tilde u(x) = u(x\sqrt{d})$ and $\tilde c = c/\sqrt{d}$. Thus, in Theorem \ref{cD} we retrieve the same asymptotic order for $c(D)$ as in the homogeneous case. The comparison between $c_0$ and $c_\infty$ is an interesting question and we wish to answer it in another paper. 
\end{rmq}

A by-product of the proof of Theorem \ref{cD} is the well-posedness for an a priori degenerate elliptic system, where the species of density $v$ would only diffuse vertically, which can be seen as an hypoellipticity result :
\begin{thmInt}
 \label{degen}
$c_\infty$ can be characterised as follows : there exists a unique $c_\infty > 0$ and $u \in \mathcal C^{2+\alpha}(\mathbb R),v\in \mathcal C^{1+\alpha/2,2+\alpha}(\Omega_L)$ with $u', \partial_x v > 0$, unique up to translations in $x$ that solve
\begin{equation}
\label{parab}
  \begin{tikzpicture}
  \draw (-6,0) -- (6,0) node[pos=0.5,below] {\small{$d\partial_y v = \mu u - v$}} node[pos=0.5,above] {$-u'' + c_\infty u' = v - \mu u$};

  \node at (4.3,0.33) {$u\to 1/\mu$};
  \node at (-4.3,0.33) {$0 \leftarrow u$};

  \node at (0,-1.5) {$ c_\infty\partial_x v - d\partial_{yy} v  = f(v)$};

  \draw (-6,-3) -- (6,-3) node[pos=0.5,above] {\small{-$\partial_y v= 0$}};


  \node at (4.3,-1.5) {$v \to 1$};
  \node at (-4.3,-1.5) {$0 \leftarrow v$};

  \end{tikzpicture}
\end{equation}
\end{thmInt}
\ \\
\noindent We will present two proofs of Theorem \ref{degen}. One by studying the asymptotic behaviour of $c(D)$ thanks to estimates in the same spirit as the ones of Berestycki and Hamel in \cite{BH05}. Another one of independent interest, by a direct method, showing that the system \eqref{parab} is not degenerate despite the absence of horizontal diffusion in the strip. Both proofs consist in showing a convergence of some renormalised profiles to a limiting profile, solution of the limiting system \eqref{parab}.


From now on, we renormalise $(c,u,v)$ in \eqref{normal} by making $$x \leftarrow \sqrt{D}x, c \leftarrow c/\sqrt{D}$$ ending up with the following equations

\begin{equation}
\label{rescaled}
  \begin{tikzpicture}
  \draw (-6,0) -- (6,0) node[pos=0.5,below] {\small{$d\partial_y v = \mu u - v$}} node[pos=0.5,above] {$-u'' + cu' = v - \mu u$};

  \node at (4.3,0.33) {$u\to 1/\mu$};
  \node at (-4.3,0.33) {$0 \leftarrow u$};

  \node at (0,-1.5) {$ -\frac{d}{D} \partial_{xx} v - d\partial_{yy} v + c\partial_x v = f(v)$};

  \draw (-6,-3) -- (6,-3) node[pos=0.5,above] {\small{$-\partial_y v= 0$}};


  \node at (4.3,-1.5) {$v \to 1$};
  \node at (-4.3,-1.5) {$0 \leftarrow v$};

  \end{tikzpicture}
\end{equation}
for which we need to show $$\lim_{D\to+\infty} c(D) = c_\infty >0$$ in order to prove Theorem \ref{cD}.

Before getting into the substance, we would like to mention that there is an important literature about speed-up or slow-down of propagation in reaction-diffusion equations in heterogeneous media and we wish to briefly present some of it.

\subsection*{Some other results}
Closest to our work is the recent paper of Hamel and Zlatoš \cite{HZ}, concerned by the speed-up of a combustion front by a shear flow. Their model is :
\begin{equation}
\label{eqhz}
\partial_t v + A\alpha(y) \partial_x v = \Delta v + f(v), \qquad t\in \mathbb R, (x,y)\in\mathbb R\times\mathbb R^{N-1}
\end{equation}
where $A > 1$ is large, and where $\alpha(y)$ is smooth and $(1,\cdots,1)$-periodic. They show that there exists $\gamma^*(\alpha, f) \geq \int_{\mathbb T^{N-1}} \alpha(y) dy$ such that the velocity $c^*(A\alpha,f)$ of travelling fronts of \eqref{eqhz} satisfies
$$\lim_{A\to+\infty} \frac{c^*(A\alpha,f)}{A} = \gamma^*(\alpha,f)$$
and under an Hörmander type condition on $\alpha$\footnote{Namely, $\alpha$ is smooth and there exists $r \in \mathbb N^*$ such that $\sum_{1 \leq |\zeta | \leq r} |D^\zeta \alpha(y) | > 0$} they characterise $\gamma^*$ as the unique admissible velocity for the following degenerate system where $ \gamma \in \mathbb R, U \in L^\infty$ and $\nabla_y~U\in~L^2~\cap~L^\infty$~:
\begin{equation}
\begin{cases}
\Delta_y U + (\gamma - \alpha(y)) \partial_x U + f(U) = 0\text{ in }\mathscr D'(\mathbb R\times \mathbb T^{N-1}) \\
0 \leq U \leq 1 \text{ a.e. in } \mathbb R\times \mathbb T^{N-1} \\
\lim_{x\to+\infty} U(x,y) \equiv 0 \text{ uniformly in }\mathbb T^{N-1} \\
\lim_{x\to-\infty} U(x,y) \equiv 1 \text{ uniformly in }\mathbb T^{N-1}
\end{cases}
\end{equation}

Let us also give a brief account of other results concerning enhancement of propagation of reaction-diffusion fronts, especially motivated by combustion modelling and in heterogeneous media. In the presence of heterogeneities, quantifying propagation is considerably more difficult than the argument of Remark \ref{rmqhomo}. The pioneering work in this field goes back to the probabilistic arguments of Freidlin and Gärtner \cite{FreGa} in 1979. They studied KPP-type propagation in a periodic environment and showed that the speed of propagation is not isotropic any more : propagation in any direction is influenced by all the other directions in the environment, and they gave an explicit formula for the computation of the propagation speed.

Reaction-diffusion equations in heterogeneous media since then is an active field and the question of the speed of propagation has received much attention. Around 2000, Audoly, Berestycki and Pommeau \cite{ABP}, then Constantin, Kiselev and Ryzhik \cite{CKR} started the study of speed-up or slow-down properties of propagation by an advecting velocity field. This study is continued in \cite{KR} and later by Berestycki, Hamel and Nadirashvili \cite{BHN} and Berestycki, Hamel and Nadin \cite{BHNa} through the study of the relation between the principal eigenvalue and the amplitude of the velocity field.

Apart from speed-up by a flow field, the influence of heterogeneities in reaction-diffusion is studied in a series of paper \cite{BHNI,BHNII} published in 2005 and 2010, where Berestycki, Hamel and Nadirashvili, following \cite{BHPerEx} gave some new information about the influence of the geometry of the domain and the coefficients of the equation. The first paper deals with a periodic environment, the second with more general domains. In 2010 also, explicit formulas for the spreading speed in slowly oscillating environments were also given for the first time by Hamel, Fayard and Roques in \cite{HFR}.

The influence of geometry on the blocking of propagation was also studied in periodic environment by Guo and Hamel \cite{JSH} and in cylinders with varying cross-section by Chapuisat and Grenier \cite{GreCha}.



The present paper highlights a totally different mechanism of speed-up by the heterogeneity, through a fast diffusion on a line.

\subsection*{Organisation of the paper}
The strategy of proof is the following : first, we show that there exists constants $0 < m < M$ independent of $D$ such that $m < c(D) < M$. Then, we show that the limit point of $c(D)$ as $D\to +\infty$ is unique and we characterise it, which proves Theorem \ref{cD} and \ref{degen}. Another section is devoted to the proof of direct existence for system \eqref{parab}. More precisely, the organisation is as follows : 

\begin{itemize}
 \item In Section \ref{upboundDinf} we compute positive exponential solutions of the linearised near $0$ of \eqref{rescaled}. Those are fundamental to study the tail of the solutions as $x\to -\infty$ for comparison purposes. We use them to show that $c(D) \leq M$.
 \item Section \ref{lowerbound} is devoted to showing that $c(D) \geq m$ by proving some integral estimates.
 \item Section \ref{equiv} proves Theorem \ref{cD} by showing the uniqueness of the limiting point $c_\infty > 0$ of $c(D)$. This uses integral identities and a mixed parabolic-elliptic sliding method.
 \item Finally, in Section \ref{direct} we construct travelling waves to the limiting system \eqref{parab} by a direct method, proving Theorem \ref{degen}. For this, we treat $x$ as a time variable and combine standard parabolic and elliptic theory.
\end{itemize}

\section{Positive exponential solutions, upper bound}
\label{upboundDinf}
We compute positive exponential solutions of \eqref{rescaled} with $f=0$. Those play an important role for comparison purposes as $x\to -\infty$ and in the construction of supersolutions. Looking for $\phi(x) = e^{\lambda x}, \psi(x,y) = e^{\lambda x} h(y)$ with $h > 0$ we get the equations

\begin{equation}
\begin{cases}
\label{eqhl}
  -h'' + \lambda\left(\frac{c}{d}-\frac{1}{D}\lambda\right)h = 0\text{ for } y\in(-L,0) \\
  h'(-L) = 0 \\
  dh'(0) =\mu - h(0) \\
  -\lambda^2 + c\lambda = h(0) - \mu
 \end{cases}
\end{equation}
Since we are interested in the asymptotic behaviour of $c(D)$, we can assume $D > d$ and get a solution given by

$$\begin{cases}
    \psi^D(x,y) = \displaystyle\frac{\mu e^{\lambda x}\cosh\left(\beta(\lambda)\left(y+L\right)\right) }{\cosh\left(\beta(\lambda) L\right) + d\beta(\lambda)\sinh\left(\beta(\lambda)L\right)} = e^{\lambda x}h^D(y)\\
    \phi^D(x) = e^{\lambda x}
  \end{cases}$$
where $\beta(\lambda) = \sqrt{\lambda(\frac{c}{d}-\frac{\lambda}{D})}$ and with $c < \lambda < c\frac{D}{d}$ solving 
\begin{equation}
\label{eqlambda}
-\lambda^2 + c\lambda = \frac{-\mu d \beta(\lambda)\tanh(\beta(\lambda)L)}{1+d\beta(\lambda)\tanh(\beta(\lambda)L)}
\end{equation}
as pictured in figure \ref{figlambdaDinf}. 
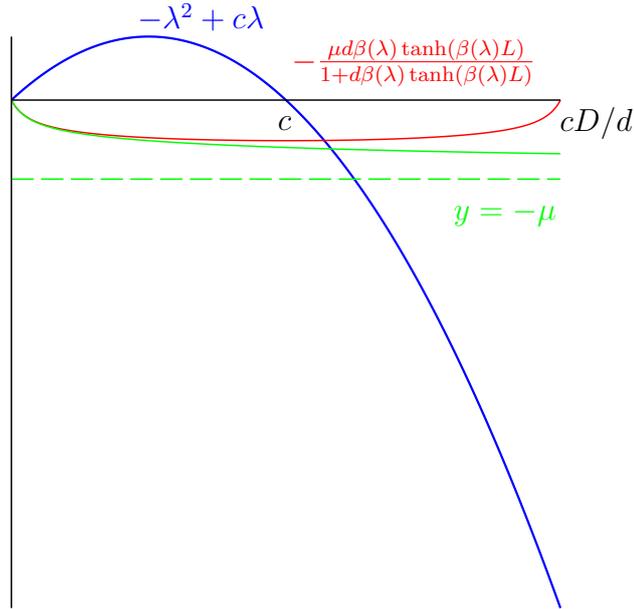
\begin{figure}[h!]
 \centering
 \includegraphics[scale=0.4,keepaspectratio=true]{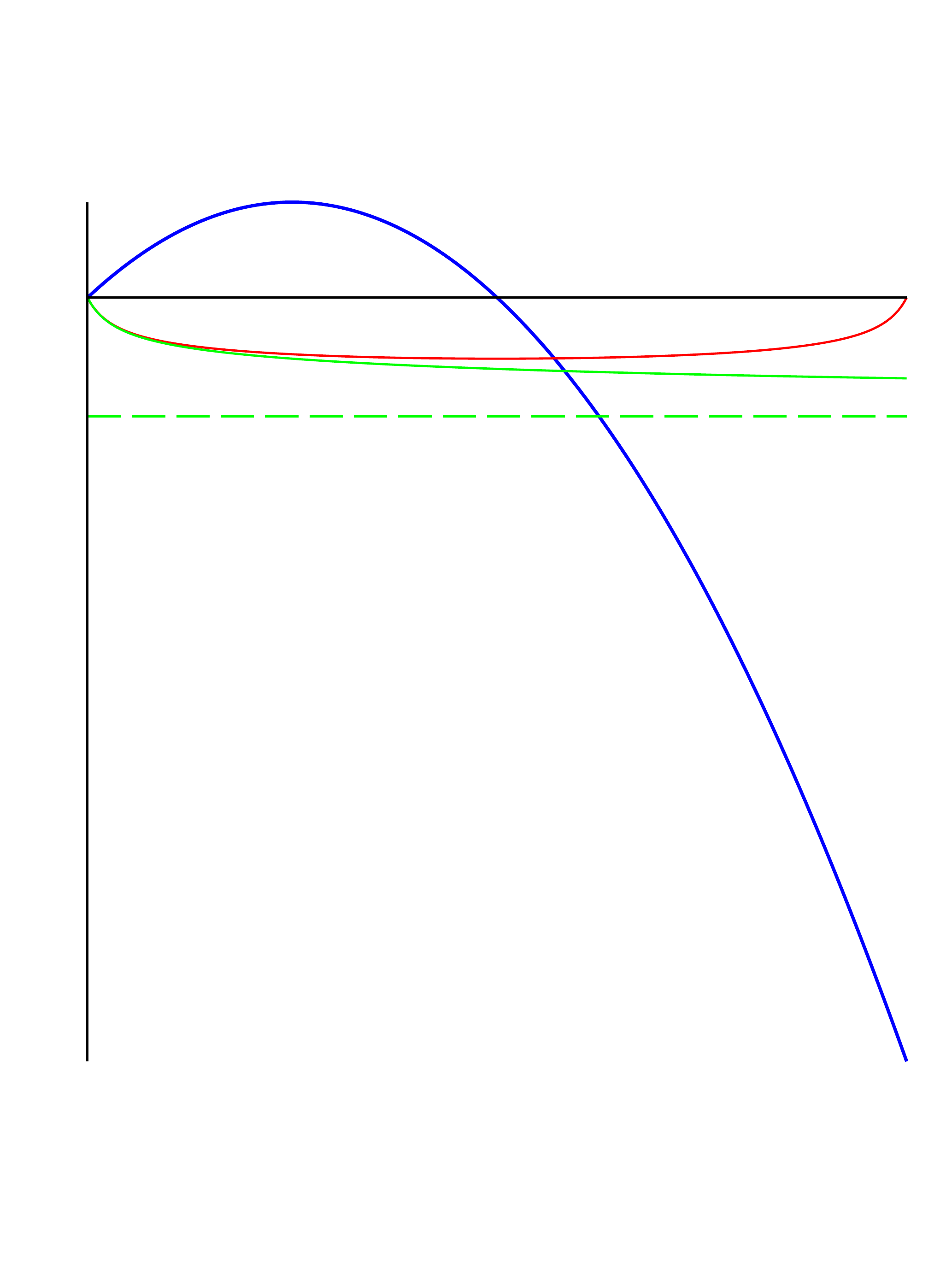}
\vspace{-40pt}
 \caption{Eq. on $\lambda$ in \eqref{eqhl}}
\label{figlambdaDinf}
    \rput(0.1,8.2){$c$}
        \rput(4.2,8.2){$cD/d$}
     \rput(-1,9.6){\textcolor{blue}{$-\lambda^2 + c\lambda$}}
    \rput(1.8,9){\textcolor{red}{$-\frac{\mu d\beta(\lambda) \tanh(\beta(\lambda) L)}{1+d\beta(\lambda) \tanh(\beta(\lambda) L)}$}}
    \rput(3,7){\textcolor{green}{$y=-\mu$}}
\end{figure}
Moreover, since the right-hand side of \eqref{eqlambda} is a decreasing function of $D$, we know that $\lambda$ is an increasing function of $D$, so $\lambda < \frac{c+\sqrt{c^2+4\mu}}{2}$ (see on Figure  \ref{figlambdaDinf} the horizontal asymptote $-\mu$ of the graph of the right-hand side of the equation when $D= + \infty$). Thus we have the uniform bounds in $D$ : $$c < \lambda < \frac{c+\sqrt{c^2+4\mu}}{2}$$

\begin{rmq}
\label{limitelambda}
Actually we have even better : since the right-hand side of \eqref{eqlambda} converges to $$\frac{-\mu\sqrt{dc\lambda}\tanh(\sqrt{d^{-1}c\lambda}L)}{1+\sqrt{dc\lambda}\tanh(\sqrt{d^{-1}c\lambda}L)}$$ as $D\to +\infty$, we know that $\lambda$ increases to the solution of 
\begin{equation*}
\label{eqlambdaDinf}
-\lambda^2 + c\lambda = \frac{-\mu\sqrt{dc\lambda}\tanh(\sqrt{d^{-1}c\lambda}L)}{1+\sqrt{dc\lambda}\tanh(\sqrt{d^{-1}c\lambda}L)}
\end{equation*}
as pictured on Figure \ref{figlambdaDinf}.
\end{rmq}

We will also keep in mind that for every $D \in (d,+\infty]$, $c\mapsto \lambda(D,c)$ is an increasing function, indeed $-\lambda^2 + c\lambda$ is increasing and the right-hand side of the equation is decreasing (it can be written $\frac{-\mu g(c)}{1+g(c)}$ with $g'(c) > 0$).


From now on, we normalise $h(y)$ such that $\min_{y\in [-L,0]} h(y) = 1$ and study the tail of the fronts as $x\to -\infty$.

\begin{prop}
\label{tail}
Let $(c,u,v)$ denote the unique solution of \eqref{rescaled} that satisfies 
\begin{equation}
\label{normalcond}
\max_{x \leq 0, y\in[-L,0]}\ (\mu u(x), v(x,y)) = \theta
\end{equation}
and call $m = \min\left( \min_{y\in[-L,0]} v(0,y), \mu u(0) \right)$. Then on $x\leq 0$, $$\frac{m}{\max h} e^{\lambda x}h(y) \leq \mu u, v \leq \theta e^{\lambda x} h(y)$$
with the notations of Section \ref{upboundDinf}.
\end{prop}

\begin{proof} 
Call $\mu \bar u =\theta e^{\lambda x}, \bar v = \theta e^{\lambda x} h(y)$. Then $\mu U := \mu (\bar u - u), V:= \bar v - v$ satisfy :

\begin{equation*}
  \begin{tikzpicture}
  \draw (-6,0) -- (4.3,0) node[pos=0.5,below] {\small{$d\partial_y V = \mu U - V(x,0)$}} node[pos=0.5,above] {$-U'' + cU' = V(x,0) - \mu U$};

  \node at (4.3,0.33) {$U \geq 0$};
  \node at (-5,0.33) {$0 \leftarrow U$};

  \node at (0,-1.5) {$-\frac{d}{D}\partial_{xx} V - d\partial_{yy} V + c\partial_x V = 0$};

  \draw (-6,-3) -- (4.3,-3) node[pos=0.5,above] {\small{$-\partial_y v= 0$}};


  \node at (5.2,-1.5) {$V \geq 0$};
  \node at (-5,-1.5) {$0 \leftarrow V$};
  
  \draw (4.3,0) -- (4.3,-3);

  \end{tikzpicture}
\end{equation*}
Suppose there is a point where $V < 0$. Since $V$ decays to $0$ uniformly in $y$ as $x \to -\infty$, $V$ reaches a negative minimum somewhere. By the normalisation condition \eqref{normalcond}, the strong maximum principle and Hopf's lemma (see \cite{GT}), it can only be on $x < 0, y=0$ and at this point we have $\mu U < \min V$.

This is a contradiction : looking at the equation on $U$, its limit as $x \to -\infty$ and its non-negative value at $x=0$, we can assert that it reaches a minimum at some $x_U < 0$ where the equation gives $\mu U(x_U) = V (x_U) + U''(x_U) \geq V(x_U) \geq \min V$. In the end, $V \geq 0$ and the maximum principle applied on $U$ gives $U \geq 0$.

The exact same argument applied on $u - \underline u, v - \underline v$ give the other inequality.
\end{proof}

\begin{prop}
\label{cmax}
There is a uniform bound in $D$ on the velocity $c(D)$ of solutions of \eqref{rescaled}~: $$c(D) \leq \sqrt{\frac{D}{D-d}\text{Lip} f} \sim_{D\to+\infty} \sqrt{\text{Lip} f} $$
\end{prop}

\begin{proof}
Call $\mu\bar u = \bar v = e^{cx}$. A simple computation shows that if $$c^2 (1-d/D) \geq \text{Lip} f$$ then $(\bar u, \bar v)$ is a supersolution of \eqref{rescaled}. We now use a sliding argument (see \cite{BN91}, \cite{V}):

Since $\lambda > c$ in Prop. \ref{tail}, we know that the graph of $e^{cx}$ is asymptotically above the ones of $\mu u$ and $v$. Knowing this and since $\mu u, v \leq 1$, we can translate the graph of $e^{cx}$ to the left above the ones of $\mu u$ and $v$. Now we slide it back to the right until one of the graphs touch, which happens since $\mu u, v \to 1$ as $x\to +\infty$ whereas $e^{cx} \to 0$ as $x\to -\infty$, uniformly in $y$. What we just said proves that
$$r_0 = \inf \{ t\in \mathbb R \mid \bar v(t + x,y) - v(x,y) > 0\textbf{ and } \bar u(t + x,y) - u(x) > 0\}$$
exists as an $\inf$ over a set that is non-void and bounded by below. Now call $$U(x):= \bar u(r_0 + x) - u(x)$$ $$V(x,y):= \bar v(r_0 + x) - v(x,y)$$ By continuity, $U,V \geq 0$. But $\mu U, V$ satisfy

\begin{equation*}
  \begin{tikzpicture}
  \draw (-6,0) -- (6,0) node[pos=0.5,below] {\small{$d\partial_y V = \mu U - V(x,0)$}} node[pos=0.5,above] {$-U'' + cU' = V(x,0) - \mu U$};

  \node at (5.3,0.33) {$U\to +\infty$};
  \node at (-5.3,0.33) {$0 \leftarrow U$};

  \node at (0,-1.5) {$-\frac{d}{D}\partial_{xx} V - d\partial_{yy} V + c\partial_x V + k(x,y)V \geq 0 $};

  \draw (-6,-3) -- (6,-3) node[pos=0.5,above] {\small{$-\partial_y v= 0$}};


  \node at (5.3,-1.5) {$V \to +\infty$};
  \node at (-5.3,-1.5) {$0 \leftarrow V$};

  \end{tikzpicture}
\end{equation*}
where $k(x,y) = -\displaystyle\frac{f(\bar v(r_0 + x)) - f(v(x,y))}{\bar v(r_0 + x) - v(x,y)} \in L^\infty$ since $f$ is Lipschitz. Using the strong maximum principle to treat a minimum that is equal to $0$ (so that no assumption on the sign of $k$ is needed) and treating the boundary $y=0$ as above, knowing that $V \not\equiv 0$ we end up with $V > 0$.

But then for any fixed compact $K_a = [-a,a]\times [-L,0]$, $\min_{K_a} V, \min_{[-a,a]} U > 0$ so that we can translate the graph of $e^{cx}$ a little bit more to the right while still being above the ones of $\mu u$ and $v$ on $K_a$, i.e. $\bar u(r_0 - \varepsilon_a + x) - u(x) > 0$ on $[-a,a]$ and $\bar v(r_0 - \varepsilon_a + x) - v(x,y) > 0$ on $K_a$ for $\varepsilon_a > 0$ small enough.

Now just chose $a$ large enough so that on resp. $x < -a$ and $x > a$, $\mu \bar u(r_0 + \varepsilon_a + x), \bar v(r_0 + \varepsilon_a + x,y), \mu u, v$ are resp. close enough to $0$ or large enough so that $k_a(x,y)=-\frac{f(\bar v(r_0-\varepsilon_a + x)) - f(v(x,y))}{\bar v(r_0 - \varepsilon_a + x) - v(x,y)}$ has the sign of $-f'(0) = 0$ or $-f'(1) > 0$. Now the maximum principle applies just like above on $x < -a$ and $x > a$ and concludes that $\bar u(r_0 - \varepsilon_a + x) - u(x), \bar v(r_0 - \varepsilon_a + x) - v(x,y) > 0$ on the whole $\mathbb R \times \Omega_L$, which is a contradiction with the definition of $r_0$.

In the end, no such $\bar u, \bar v$ can exist, i.e. $c^2 (1-d/D) \leq \text{Lip} f$.

\end{proof}

\begin{rmq}
This proof shows how rigid the equations of fronts are when involving a reaction term with $f'(0), f'(1) \leq 0$ : it is shown in \cite{LD14} that there is no supersolution or subsolution (in a sense defined in \cite{LD14}) except the solution itself and its translates. This fact was already noted in \cite{BN90, V} for Neumann boundary value problems.
\end{rmq}

\section{Proof of the lower bound}
 \label{lowerbound}
 In this section we show the following : 
 \begin{prop}
$$\inf_{D>d} c(D) = c_{min} > 0$$
 \end{prop}
We proceed by contradiction. Suppose that $\inf c(D) = 0$. Then there exists a sequence $D_n \to \infty$ (since $c$ is a continuous function of $D$, see \cite{LD14}) such that the associated solutions $(c_n,\phi_n,\psi_n)$ satisfy $c_n \to 0$. 
Moreover, integrating by parts the equation on $v$ in \eqref{rescaled} and using elliptic estimates to assert $u', \partial_x v \to 0$ as $x\to\pm\infty$ we get $$c_n=\frac{1}{L+1/\mu} \int_{\Omega_L} f(\psi_n)$$ so we know that $\int_{\Omega_L} f(\psi_n) \to 0$ which also gives $$\int_{\Omega_L} f(\psi_n)\psi_n \leq \int_{\Omega_L} f(\psi_n) \to 0$$ Multiplying the equation by $\psi_n$ and integrating by parts yields

\begin{equation}
\label{ipph1}
\frac{d}{D_n} \int_{\Omega_L}\left(\partial_x\psi_n\right)^2 + d\int_{\Omega_L} \left(\partial_y\psi_n\right)^2  + \int_{\mathbb R} \phi_n'\partial_x\psi_n(\cdot,0) + c_n\int_{\mathbb R}\phi_n'\psi_n(\cdot,0) + \frac{c_nL}{2} = \int_{\Omega_L} f(\psi_n)\psi_n\end{equation}
All the terms in the left hand side of this expression are positive quantities, so each one of them must go to zero as $n \to\infty$. 
Now, we normalise $\psi_n$ by $$\psi_n(0,0) = \theta_1 \in ]\theta, 1[$$ and assert the following :

\begin{lem}
\label{lemdelta}
 Fix $\delta > 0$ small. There exists $N > 0$ such that for all $n > N$ we have for all $ -1 \leq x \leq 1$ : $$ \left(1-\frac{\delta}{2}\right)\theta_1 < \psi_n(x,0) < \left(1+\frac{\delta}{2}\right)\theta_1 $$ 
\end{lem}

Before giving the proof, we mention an easy but technical lemma that will be used :

\begin{lem}
\label{lemfourier}
 If $k\in L^1$, $\hat k\in\mathcal C^\infty \cap L^2$ and $h\in L^\infty$ then the formula $$\mathcal F^{-1}(\hat k \hat h) = k*h$$ makes sense and holds.
\end{lem}
\begin{proof}
Since $\hat k$ is a smooth function, the product distribution $\hat k \hat h$ makes sense and we can compute its inverse Fourier transform : we leave it to the reader to check the result using the classical properties of the Fourier transform on $L^2$ and the Fubini-Tonelli theorem.
\end{proof}

We now turn to the proof of lemma \ref{lemdelta}.
\begin{proof}
 We know that $\int_{\mathbb R} \phi_n'\partial_x\psi_n(\cdot,0) \to 0$. Note that $$-\phi_n'' + c_n\phi_n' +  \mu\phi_n = \psi_n(\cdot,0)$$ and so by lemma \ref{lemfourier} and the fact that $\xi^2 - c_ni\xi + \mu$ has no real roots, we have $\phi_n = K_n*\psi_n(\cdot,0)$ where $\hat{K_n}(\xi) = \frac{1}{\xi^2 - c_ni\xi + \mu}$, i.e. $$K_n(x) = \sqrt{\frac{2\pi}{c_n^2 + 4\mu}} e^{  -\frac{1}{2}\left(\sqrt{c_n^2+4\mu}-c_n\right)x }\left(e^{\sqrt{c_n^2+4\mu}x}H(-x) + H(x)\right)$$ This is a sequence of positive functions, uniformly bounded from below by a positive constant on any compact subset of $\mathbb R$.

\begin{figure}[h!]
\vspace{-190pt}
 \centering
 \includegraphics[scale=0.7,keepaspectratio=true]{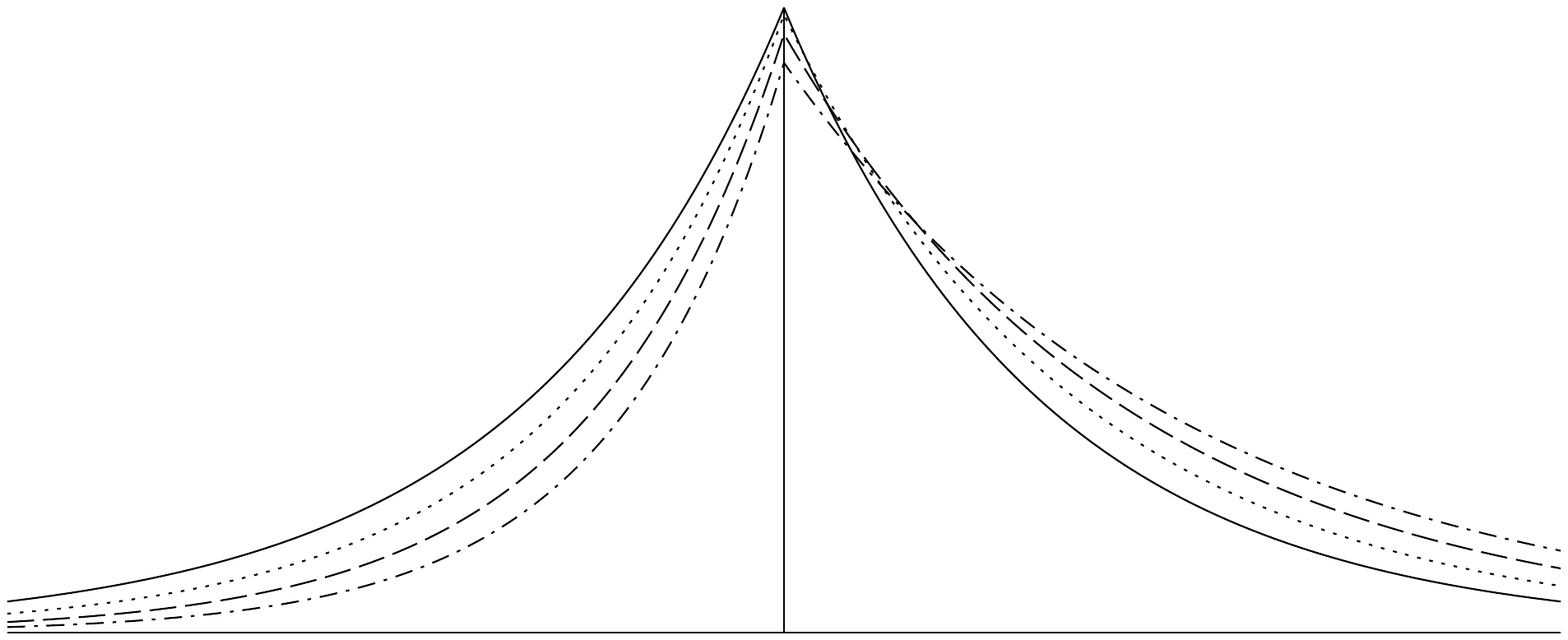}
  \rput(6.7,7.5){$x$}
  \rput(0,7.5){$0$}
  \rput(-6,8.8){$\sqrt{\frac{\pi}{2\mu}}e^{-\mu|x|}$}

\vspace{-210pt}
 \caption{Graphs of $K_n$ for $\mu=1$, $c_n \in \{0, 0.3, 0.6, 0.9\}$}
 \label{seqkn}
\vspace{10pt}
\end{figure}

Now for $-1 \leq x \leq 1$ and since $\partial_x\psi_n \geq 0$ we have $$|\psi_n(x,0) - \theta_1| \leq \int_{-1}^1 \partial_x\psi_n(\cdot,0)$$
But since $K_n, \partial_x\psi_n \geq 0$ and since $K_n > \alpha_0 > 0$ on $[-2,2]$ :

\begin{equation*}
\begin{split}
  \int_{\mathbb R} (\partial_x\psi_n) \phi_n' \geq \int_{\mathbb R} (K_n*\partial_x\psi_n)\partial_x\psi_n \geq & \int_{-1}^1 \left(\int_{\mathbb R} K_n(x-t)\partial_x\psi_n(t,0) dt\right) \partial_x\psi_n(x,0) dx \\ \geq & \int_{-1}^1 \left(\int_{-1}^1 K_n(x-t)\partial_x\psi_n(t,0) dt\right) \partial_x\psi_n(x,0) dx \\ >&\ \alpha_0 \left(\int_{-1}^1 \partial_x\psi_n(\cdot,0)\right)^2
\end{split}
\end{equation*}
Thus $$\int_{-1}^1 \partial_x\psi_n(\cdot,0) \leq \left(\frac{1}{\alpha_0}\int_{\mathbb R} \left(\partial_x\psi_n\right) \phi_n'\right)^{1/2} \to 0$$ and for $n$ large enough this quantity is less than $\delta \theta_1/2$.

\end{proof}

\begin{lem}
 \label{lemborel}
 Fix $0 < \delta < \frac{2\sqrt{2L}}{\theta_1}$. There exists $N' > 0$ such that for all $n > N'$, there exists a borelian $J_n$ of $[-1,1]$ with measure $\geq 1$ such that for all  $x\in J_n$ : $$\left(\int_{-L}^0 \partial_y\psi_n(x,s)^2 ds\right)^{1/2} \leq \frac{\delta\theta_1}{2\sqrt{L}}$$
\end{lem}
\begin{proof}
 This result is based on a Markov-type inequality. Call $$h(x):= \int_{-L}^0 \partial_y\psi_n(x,s)^2ds$$
 We now use that $$0 \leq \int_{-1}^1 h(x) dx \leq \int_{\mathbb R} \left( \int_{-L}^0 \partial_y\psi_n(x,s)^2ds\right) dx \to 0$$ Thus for all $\varepsilon > 0$, there exists $N' > 0$ such that for all $n > N'$, $\int_{-1}^1 h(x) dx \leq \varepsilon$. Using then that $h(x) \geq \sqrt{\varepsilon}\mathbf{1}_{h > \sqrt{\varepsilon}}$ we get that

$$ |\{x\in[-1,1] \mid h(x) > \sqrt{\varepsilon}\}| < \frac{1}{2}\sqrt{\varepsilon}$$
where $|A|$ denotes the Lebesgue measure of $A$. We get the result by choosing $$J_n = \{x\in[-1,1] \mid h(x) \leq \sqrt{\varepsilon}\}$$ and $\varepsilon = \delta^4\theta_1^4/16L^2$. The fact that $|J_n| \geq 1$ directly comes from the upper bound assumed on $\delta$.
\end{proof}

We can now finish the proof. Indeed, for $n > N,N'$ fixed and for $x\in J_n$ we have 
$$ \psi_n(x,0) - \int_{-L}^0 |\partial_y(x,s)ds| \leq \psi_n(x,y) \leq \psi_n(x,0) + \int_{-L}^0 |\partial_y(x,s)ds|$$
Using the Cauchy-Schwarz inequality and Lemmas \ref{lemdelta} and \ref{lemborel}, we get on $J_n\times[-L,0]$ 
$$ (1-\delta)\theta_1 \leq \psi_n(x,y) \leq (1+\delta)\theta_1 $$
And so $$\left(1+\frac{L}{\mu}\right)c_n = \int_{\Omega_L} f(\psi_n) \geq \int_{J_n\times[-L,0]} f(\psi_n) \geq L \inf_{J_n\times[-L,0]} f(\psi_n)$$ Choosing now $\theta_1$ and $\delta$ such that $f([(1-\delta)\theta_1,(1+\delta)\theta_1  ]) > C > 0$ we get a contradiction with the assumption $c_n \to 0$.

\section{The equivalent $c(D) \sim c_\infty\sqrt{D}$}
\label{equiv}
We know that every sequence $c(D_n)$ associated to a sequence $D_n \to +\infty$ is trapped between two positive constants.  Now we just have to show the uniqueness of the limit point. We divide the proof in three steps.
\begin{itemize}
 \item Compactness : we prove that any $(\phi_n,\psi_n)$ associated to $D_n \to \infty$ and $c_n \to c > 0$ is bounded in $H^3_{loc}$. This uses integral identities.
 \item Treating the $x$ variable as a time variable, we extract from such a family a $(c,\phi,\psi)$ that solves \eqref{rescaled} with $D = +\infty$.
 \item We show uniqueness of $c$ for such a problem using a parabolic version of the arguments in Proposition \ref{cmax}.
\end{itemize}

But first, let us give an easy but technical lemma that will be used in the next computations.

\begin{lem}{Gagliardo-Nirenberg and Ladyzhenskaya type inequalities in $\Omega_L$} 
\label{GN}
\begin{itemize}
 \item For all $\alpha\in ]\frac{1}{3},1[$, there exists $C_{GN} > 0$ s.t. for all $u\in H^1(\Omega_L)$ $$|u|_{L^3(\Omega_L)} \leq C_{GN}|u|_{L^2(\Omega_L)}^{1-\alpha} |u|_{H^1(\Omega_L)}^\alpha$$
 \item For all $\beta\in ]\frac{1}{2},1[$, there exists $C_L > 0$ s.t. for all $u\in H^1(\Omega_L)$ $$|u|_{L^4(\Omega_L)} \leq C_L|u|_{L^2(\Omega_L)}^{1-\beta} |u|_{H^1(\Omega_L)}^\beta$$
\end{itemize}
\end{lem}
\begin{proof}
 The inequalities above with $\alpha = \frac{1}{3}$ and $\beta = \frac{1}{2}$ are resp. Gagliardo-Nirenberg and Ladyzhenskaya inequalities. We can prove that these are still valid in $\Omega_L$ by re-doing the computations of Nirenberg and using trace inequalities, but the inequalities above will suffice for us and are easier to prove. 
 
 For this, we just use the Hölder interpolation inequality : if $p,q,r \geq 1$ and $\alpha\in[0,1]$  such that $\frac{1}{q} = \frac{1-\alpha}{r} + \frac{\alpha}{s}$, then $|u|_{L^q} \leq |u|_{L^r}^{1-\alpha}|u|_{L^s}^\alpha$. We apply this with $q=3$ resp. $4$, $r=2$, $s \to +\infty$, and control the $|u|_{L^s}^\alpha$ part with the Sobolev embedding $|u|_{L^s}^\alpha \leq C|u|_{H^1}^\alpha$ for $s\geq 2$.
\end{proof}

We are now able to treat Step 1 :

\begin{lem}
\label{H3estimate}
Let $D_n \to \infty$, $(c_n,\phi_n,\psi_n)$ the sequence of solutions associated to $D_n$ and suppose $c_n \to c > 0$. Denote $\Omega_{L,M} := [-M,M]\times[-L,0]$. Then for every $M\in\mathbb N$, there exists $C_M > 0$ s.t. $|\psi_n|_{H^3(\Omega_{L,M})} \leq C_M$.
\end{lem}

\begin{proof}

\begin{enumerate}[a)]
\item The $H^1$ bound. 

First, since $0 \leq \psi_n \leq 1$, we know that $(\psi_n)$ is bounded in $L^2(\Omega_{L,M})$. Then we start from \eqref{ipph1} : since $(c_n)$ is bounded and because the right-hand side of \eqref{ipph1} is bounded, we have that $(\partial_y\psi_n)$ is bounded in  $L^2(\Omega_L)$. \\ Then multiply the equation inside $\Omega_L$ in \eqref{rescaled}  by $\partial_x \psi_n$ and in a similar fashion as \eqref{ipph1}, integrate by parts on $\Omega_{L,M}$. Boundary terms along the $y$ axis decay thanks to elliptic estimates. Using $\partial_x\psi_n > 0$ and $|\phi_n|_{W^{2,\infty}(\mathbb R)} \leq C$ (use Fourier transform of the variation of constants) we get the sum of following terms

\begin{itemize}
\item $\displaystyle-\frac{d}{D_n}\int_{\Omega_L}\partial_{xx}\psi_n \partial_x\psi_n = 0$
\item $\displaystyle\int_{\Omega_L} -d\partial_{yy}\psi_n\partial_x\psi_n = \cancel{\int_{\Omega_L}\partial_x\left( \frac{1}{2}d\partial_y\psi_n^2\right) }+ \int_{\mathbb R}(-\phi_n'' + c_n\phi_n')\partial_x\psi_n$ so $$\left|\int_{\Omega_L} -d\partial_{yy}\psi_n\partial_x\psi_n\right| \leq C$$
\item $\displaystyle\left|\int_{\Omega_L}f(\psi_n)\partial_x\psi_n\right| \leq \sup f\times L$
\item Finally, $\displaystyle c_n\int_{\Omega_L}\left(\partial_x\psi_n\right)^2$, which we want to bound.
\end{itemize}
so that in the end $$\displaystyle\left|\partial_x\psi_n\right|_{L^2(\Omega_L)}^2 \leq \frac{L\sup f + C}{c_{min}} =: C_1$$

\item The $H^2$ bound

We apply the same process, on the equation satisfied by $(z_n,w_n):=(\phi_n',\partial_x\psi_n)$. The linear structure is the same, but this time we do not have positivity of the first $x$ derivative any more. Multiplying by $w_n:=\partial_x\psi_n  > 0$ the equation satisfied by $w_n$ and integrating gives rise to the following terms :

\begin{itemize}
\item $\displaystyle\int_{\Omega_L} \left(-\frac{d}{D_n} \partial_{xx}w_n\right)  w_n = \frac{d}{D_n}\int_{\Omega_L}\left(\partial_x w_n\right)^2 \geq 0$
\item $\displaystyle c_n\int_{\Omega_L} \left(\partial_x w_n\right) w_n = 0$ thanks to elliptic estimates.
\item $\displaystyle\left|\int_{\Omega_L} f'(\psi_n)w_n^2\right| \leq |f'|_\infty C_1$ by the result above.
\item $\displaystyle\int_{\Omega_L}\left(-d\partial_{yy}w_n\right) w_n = \int_{\Omega_L} d\left(\partial_y w_n\right)^2 + \int_{\mathbb R} (w_n - \mu z_n)w_n$ where $$\int_{\mathbb R} (w_n - \mu z_n)w_n \geq \int_{\mathbb R} w_n^2 - \mu C \geq -\mu C$$
\end{itemize}
so that in the end $$\left|\partial_y w_n\right|_{L^2(\Omega_L)}^2 \leq \frac{|f'|_\infty C_1 + \mu C}{d} =: C_2$$
We now multiply by $\partial_x w_n$ the equation satisfied by $w_n$. Integration yields the sum of the following terms :
\begin{itemize}
\item $\displaystyle\int_{\Omega_L} \left(-\frac{d}{D_n} \partial_{xx}w_n\right) \partial_x w_n = 0$
\item $\displaystyle c_n\int_{\Omega_L} \left(\partial_x w_n\right)^2$  which we want to bound.
\item $\displaystyle\int_{\Omega_L} (-d\partial_{yy}w_n) \partial_x w_n = \int_{\mathbb R} (w_n- \mu z_n)\partial_x w_n = \mu \int_{\mathbb R} z_n' w_n$ so that  $$\left|\int_{\Omega_L} (-d\partial_{yy}w_n) \partial_x w_n \right| \leq \mu C$$
\item $\displaystyle\int_{\Omega_L} f'(\psi_n)w_n \partial_x w_n = -\frac{1}{2}\int_{\Omega_L} f''(\psi_n) w_n^3$ so that 

$$\left|\int_{\Omega_L} f'(\psi_n)w_n \partial_x w_n \right| \leq \frac{|f''|_\infty}{2} |w_n|_{L^3}^3 \leq  \frac{|f''|_\infty}{2} C_{GN}^3\sqrt{C_1}^{2-\varepsilon}|w_n|_{H^1(\Omega_L)}^{1+\varepsilon}  =: C_3 |w_n|_{H^1}^{1+\varepsilon}  $$

for a small $\varepsilon > 0$ of our choice thanks to the Gagliardo-Nirenberg type inequality of Lemma \ref{GN}.
\end{itemize}
Finally, we end up with the inequality $$|w_n|_{H^1(\Omega_L)}^2 \leq C_1 + C_2 + \mu C + C_3|w_n|_{H^1(\Omega_L)}^{1+\varepsilon}$$ which yields the bound $$|w_n|_{H^1}^2 \leq C_4$$ Then we use the original equation to assert that $\partial_{yy}\psi_n$ is bounded in $L^2(\Omega_L)$, so that in the end the lemma is true with $H^2_{loc}$.

\item The $H^3$ bound

For the $H^3$ estimate, we iterate one last time with the equation satisfied by $(\tau_n,\rho_n):=(\phi'',\partial_{xx}\psi_n)$. Multiplying the equation by $\rho_n$ and integrating gives as before a non-negative and a zero term, and the boundary integral as well as the right-hand side have to be studied more carefully :

\begin{itemize}
\item $\int_{\Omega_L}  \left(-d\partial_{yy}\rho_n\right) \rho_n = \int_{\Omega_L} d\left(\partial_y \rho_n\right)^2 + \int_{\mathbb R} \rho_n^2 - \mu  \int_\mathbb R\tau_n\rho_n$. But $$\left|\int_{\mathbb R} \rho_n\tau_n\right| = \left|\int_{\mathbb R} \tau_n' \partial_x\psi_n\right| \leq |\tau_n'|_{L^2} |\partial_x\psi_n(\cdot,0)|_{L^2} \leq C_5$$ by elliptic estimates, since the $\phi_n'$ satisfy an uniformly elliptic equation on $\mathbb R$ with uniformly bounded coefficients and with data $\partial_x\psi_n(\cdot,0)$ bounded in $H^{1/2}$, so $(\phi_n)'$ is bounded in $H^{2+1/2}$, which gives $(\tau_n')$ bounded in $H^{1/2}$ so in $L^2$ (this can be seen easily through the Fourier transform). The second factor is bounded by the above result and the continuity of the trace operator on $\Omega_L$.
\item The right-hand side is $f''(\psi_n)\left(\partial_x\psi_n\right)^2 \rho_n + f'(\psi_n) \rho_n^2$. For the first term, we use that
\begin{alignat*}{1}
\left|\int_{\Omega_L}  (\partial_x\psi_n)^2\rho_n\right| &\leq \left(\int_{\Omega_L}  \left(\partial_x\psi_n\right)^4\right)^{1/2}\left(\int_{\Omega_L}  \rho_n^2\right)^{1/2} \\ 
&\leq C_L^2|\partial_x \psi_n|_{L^2}^{1-\varepsilon}|\partial_x \psi_n|_{H^1}^{1+\varepsilon}|\rho_n|_{L^2} \\ 
&\leq C_L^2 \sqrt{C_1}^{1-\varepsilon}\sqrt{C_4}^{2+\varepsilon} =: C_6
\end{alignat*}

for some small $\varepsilon > 0$ of our choice thanks to the Ladyzhenskaya type inequality of Lemma \ref{GN}. 

For the second term we just have $$\left|\int_{\Omega_L}  f'(\psi_n)\rho_n^2\right| \leq |f'|_\infty C_4$$ so this procedure gives 
$$|\partial_y \rho_n|_{L^2(\Omega_L)}^2 \leq \frac{C_5 + C_6 + |f'|_\infty C_4}{d} =: C_7$$
\end{itemize}
Now multiplying by $\partial_x \rho_n$ and integrating gives a zero term, the $\int \left(\partial_x\rho_n\right)^2$ term we want to bound, and a boundary integral as well as a right hand side that are the following :

\begin{itemize}
 \item $\int_\mathbb R (\rho_n - \mu\tau_n)\partial_x\rho_n = -\mu\int_\mathbb R \tau_n\partial_x\rho_n = -\mu\int_\mathbb R \phi_n''\partial_{xxx}\psi_n$. We bound this term thanks to a ``fractional integration by parts'' : indeed, we need to transfer more than one derivative on $\phi$, but we do not control $\phi$ in $H^4$, only in $H^{3+1/2}$. For this, we use Plancherel's identity :
 \begin{alignat*}{1}
  & \left|\int_\mathbb R \phi_n''\partial_{xxx}\psi_n\right|  = \left|\int_\mathbb R -\xi^2 \hat\phi_n i\xi^3 \hat \psi_n\right| \\ & \leq \int_\mathbb R |\xi|^{2+1/2}|i\xi\hat\phi_n| |\xi|^{1/2}|i\xi\hat\psi_n| \\ &\leq \left(\int_\mathbb R |\xi|^{2(2+1/2)}|i\xi\hat\phi_n|^2\right)^{1/2}\left(\int_\mathbb R |\xi|^{2(1/2)}|i\xi\hat\psi_n|^2\right)^{1/2} \\ &\leq \left(\int_\mathbb R \left(1+|\xi|^2\right)^{2+1/2}|i\xi\hat\phi_n|^2\right)^{1/2}\left(\int_\mathbb R \left(1+|\xi|^2\right)^{1/2}|i\xi\hat\psi_n|^2\right)^{1/2} \\ &\leq |\phi_n'|_{H^{2+1/2}} |\partial_x\psi_n(\cdot,0)|_{H^{1/2}} \\ &\leq C C_{tr}\sqrt{C_4}
 \end{alignat*}
where $C_{tr}$ is a bound for the trace operator on $\Omega_L$.

  \item The right-hand side is $\int f''(\psi_n)\left(\partial_x\psi_n\right)^2 \partial_x \rho_n + f'(\psi_n) \rho_n \partial_x\rho_n$. The first term is controlled thanks to Cauchy-Schwarz and Ladyzhenskaya's inequality again~: 
  \begin{alignat*}{1}
  \left|\int_{\Omega_L}  f''(\psi_n)\left(\partial_x\psi_n\right)^2 \partial_x \rho_n \right| \leq |f''|_\infty \int_{\Omega_L} |2\rho_n \partial_x\psi_n \partial_{xx}\psi_n| & \leq 2|f''|_\infty |\rho_n|_{L^2} |\partial_x\psi_n \partial_{xx}\psi_n|_{L^2} \\ &\leq 2|f''|_\infty |\rho_n|_{L^2}|\partial_x\psi_n|_{L^4}|\partial_{xx}\psi_n|_{L^4} \\ &\leq
  2|f''|_\infty C_L^2 \sqrt{C_4}^{5/2-\varepsilon}|\rho_n|_{H^1(\Omega_L)}^{1/2 + \varepsilon} \\ &=: C_8|\rho_n|_{H^1(\Omega_L)}^{1/2 + \varepsilon}
  \end{alignat*}
by applying Ladyzhenskaya's inequality twice.

The last term gives $\int_{\Omega_L}  f'(\psi_n)\rho_n \partial_x\rho_n = - \int_{\Omega_L} \frac{1}{2}\rho_n^2f''(\psi_n)\partial_x\psi_n$ so by Cauchy-Schwarz~:
\begin{alignat*}{1}
\left|\int_{\Omega_L}  f'(\psi_n)\rho_n \partial_x\rho_n \right| &\leq \frac{|f''|_\infty}{2}\left(\int_{\Omega_L}  |\rho_n|^4\right)^{1/2}|\partial_x \psi_n|_{L^2(\Omega_L)} \\
&\leq \frac{|f''|_\infty}{2}|\rho_n|_{L^4}^2  \sqrt{C_1}\\
&\leq \frac{|f''|_\infty}{2}C_L^2 \sqrt{C_1} |\rho_n|_{L^2}^{1 - \varepsilon} |\rho_n|_{H^1}^{1 + \varepsilon} \\
&\leq \frac{|f''|_\infty}{2}C_L^2 \sqrt{C_1} \sqrt{C_4}^{1 - \varepsilon} |\rho_n|_{H^1}^{1 + \varepsilon} \\
& := C_9|\rho_n|_{H^1}^{1 + \varepsilon}
 \end{alignat*}
by Ladyzhenskaya's inequality again.

As before, in the end we get $$|\rho_n|_{H^1}^2 \leq C_4 + C_7 + CC_{tr}\sqrt{C_4} + C_8|\rho_n|_{H^1(\Omega_L)}^{1/2 + \varepsilon} + C_9|\rho_n|_{H^1}^{1 + \varepsilon}$$ which yields the boundedness of $|\rho_n|_{H^1}$. 

Finally, the terms $\partial_{xyy} \psi_n$ and $\partial_{yyy} \psi_n$ are bounded in $L^2(\Omega_{L})$ thanks to the equation : if we differentiate the original equation on $\psi_n$ in $y$ and then in $x$, the result is immediate.
\end{itemize}
\end{enumerate}
\end{proof}


We could go on again to $H^4$ by looking at the third derivatives, but the right-hand side would involve too many computations and interpolation inequalities. Instead, we stop here and use the following lemma.

\begin{lem}
  \label{regulparab}
 With the assumptions of Lemma \ref{H3estimate}, there exists $(\phi,\psi)\in\mathcal C^2(\mathbb R)\times \mathcal C^{1_x,2_y}(\Omega_L)$ such that $(c,\phi,\psi)$ satisfies \eqref{parab} and $\phi',\partial_x\psi \geq 0$.
\end{lem}
\begin{proof}

Thanks to Lemma \ref{H3estimate}, $(\psi_n)$ is bounded in $H^3_{loc}$ so by Rellich's theorem we can extract from it a sequence that converges strongly in $H^2_{loc}$ to $\psi \in H^3_{loc}$. Moreover, thanks to the Sobolev embedding $H^3(\Omega_{L,M}) \hookrightarrow \mathcal C^{1,\gamma}$ for every $0 <\gamma < 1$, by Ascoli's theorem and the process of diagonal extraction we can assume that $\psi_n$ converges in $\mathcal C^{1,\beta}$ to $\psi \in \mathcal C^{1,\beta}$ for some $0<\beta<1$ fixed. By elliptic estimates, $(\phi_n)$ is bounded in $\mathcal C^{3,\gamma}(\mathbb R)$ for every $0<\gamma<1$, so again, we can still extract and assume that $\phi_n$ converges to a $\phi \in \mathcal C^{3,\beta}$ in the $\mathcal C^{3,\beta}$ norm.

 Since $f$ satisfies $f(0) = 0$ and is Lipschitz continuous, we can assert that $f(\psi_n)$ converges to $f(\psi)$ in $H^2_{loc}$. Then we can pass to the $L^2$ limit $n\to\infty$ in equation \eqref{rescaled} satisfied by $(c_n,\phi_n,\psi_n)$ and see that \eqref{parab} is satisfied a.e. Moreover, $\phi'$ and $\partial_x\psi$ are non-negative as locally Hölder limits of positive functions. Finally, if we fix $x_0 \in \mathbb R$ we assert that $\psi(x_0,\cdot) \in L^2(-L,0)$ and $\psi \in \mathcal C^1([x_0,+\infty[,L^2(]-L,0[) \cap \mathcal C^0(]x_0,+\infty[,H^2(]-L,0[)$. Indeed, this comes from $\psi \in H^3_{loc}$ and Jensen's inequality. For instance for $x>x_0$ and $h$ small :

  \begin{alignat*}{1}
    \displaystyle\left|\partial_{yy}\psi(x+h,\cdot)-\partial_{yy}\psi(x,\cdot)\right|_{L^2(-L,0)} &= \left(\int_{-L}^0
    \left(\partial_{yy}\psi(x+h,y)-\partial_{yy}\psi(x,y)\right)^2 dy\right)^{1/2} \\
    & = \left(\int_{-L}^0 \left( \int_x^{x+h} \partial_{xyy}\psi(s,y)ds \right)^2 dy\right)^{1/2} \\
    & \leq \left(\int_{-L}^0 \left( \int_x^{x+h} \partial_{xyy}\psi(s,y)^2 ds \right) dy\right)^{1/2} \\
    & \leq |\psi|_{H^3(]x,x+h[\times[-L,0])} \to 0
  \end{alignat*}
as $h\to 0$ as the integral of an integrable function over a set whose measure tends to zero.
It is known (see \cite{BRE}, Section 10) that such a solution is unique and has $\mathcal C^{1_x,{3+\gamma}_y}$ regularity on every $[x_0 + \varepsilon,+\infty[\times[-L,0]$. Since we can do this for every $x_0 \in \mathbb R$, the regularity announced in the lemma is proved.
The uniform limit to the left is obtained thanks to Prop. \ref{tail} : on $x \leq 0$
$$\mu \phi_n, \psi_n \leq \theta e^{\lambda_n x}h_n(y) \leq \theta e^{mx}h_- $$
where $h_- > 0$ is a uniform lower bound on $h_n$ whose existence is proved in the next lemma.

The right limits are obtained in a similar fashion as in \cite{BLL90} by integration by parts and by using standard parabolic estimates instead of elliptic ones. See Lemma \ref{energy} and Prop. \ref{rightlimits} in the next section for similar and complete computations.
\end{proof}

We conclude this section with the following lemmas, that prove the uniqueness of the limit point $c$.

\begin{lem}
\label{ordre}
Suppose $c$ and $\bar c > c$ are two limit points of $c(D)$ and denote $(c_n,\phi_n,\psi_n)$ and $(\bar c_n,\underline \phi_n,\underline \psi_n)$ some associated sequences of solutions that converge to $(c,\phi,\psi)$ and $(\bar c,\underline \phi,\underline \psi)$ as in the previous theorem. Then there exists $X\in\mathbb R$ and $N\in \mathbb N$ s.t. for all $x\leq X$ and $n\geq N$, $\underline \psi_n(x,y) < \psi_n(x,y)$.
\end{lem}

\begin{proof}
 This relies on comparison with exponential solutions computed in section \ref{upboundDinf} and on the uniform convergence of $\psi_n$ resp. $\underline \psi_n$ to $\psi$ resp. $\underline \psi$. Indeed, if as in section \ref{upboundDinf} we denote $\bar \lambda_n$ and $\underline h_n$ the exponent and the $y$ part of the exponential solutions, we claim that~: 
 $$\exists \underline h_-, \underline h_+ > 0 \mid \underline h_- < \underline h_n(y) < \underline h_+$$
 Indeed, since $\bar c_n \to \bar c$, there exists $\bar c_+ > 0$ s.t. $c_n < \bar c_+$. Then $$\beta_n(\bar\lambda_n) < \sqrt{\frac{\bar c_+\left(\bar c_+ +\sqrt{\bar c_+^2+4\mu}\right)}{2d}} =: \underline\beta_+$$ so that $$\underline h_n(y) \leq \mu\cosh(\underline\beta_+ L) =: \underline h_+$$ and $$\underline h_n(y) \geq \frac{\mu}{\cosh(\underline\beta_+ L) + d\underline\beta_+\sinh(\underline\beta_+ L)} =: \underline h_-$$ The same holds for $h_n(y)$ with constants $h_+, h_- > 0$.

 Now normalise $\underline h_n$ s.t. $\min_{[-L,0]} \underline h_n = 1$. Then we have $\underline h_n \leq \frac{\underline h_+}{\underline h_-}$, and Prop. \ref{tail} yields on $x\leq 0$ :
$$\mu \underline \phi_n, \underline \psi_n \leq \theta e^{\bar \lambda_n x}\frac{\underline h_+}{\underline h_-}$$
On the other hand, there exists $N \in \mathbb N$ s.t. for $n\geq N$ $$m_n = \min\left( \min_{y\in[-L,0]} \psi(0,y), \mu \phi(0) \right) \geq \frac{1}{2}\min\left( \min_{y\in[-L,0]} \psi(0,y), \mu \phi(0) \right) =: m$$ so that using Prop. \ref{tail} again on $x \leq 0$, for $n \geq N$,
$$\mu \phi_n, \psi_n \geq m e^{\lambda_n x}\frac{h_-}{h_+}$$

Finally, by monotonicity of $\lambda$ (see Remark \ref{limitelambda}) $\lambda_n \to \lambda(+\infty, c)$ and $\bar \lambda_n \to \lambda(+\infty,\bar c) > \lambda(+\infty, c)$, so for $n \geq N$ large enough, $\overline\lambda_n - \lambda_n > d:=\frac{1}{2}(\lambda(+\infty,\bar c) - \lambda(+\infty,c)) > 0$. Now for $$x < \frac{1}{d} \ln\left(\frac{m h_- \underline h_-}{\theta h_+ \underline h_+}\right) =: X$$ we have the inequality announced.

\end{proof}

\begin{lem}
 If $(c,\phi, \psi)$ and $(\bar c, \underline \phi, \underline \psi)$ solve the equations \eqref{parab} in the conclusion of Lemma \ref{regulparab}, then $c = \bar c$.
\end{lem}
\begin{proof}
 Since these solutions have classical regularity, we can apply the strong parabolic maximum principle and the parabolic Hopf's lemma (see for instance \cite{FRI}) in a similar fashion as the elliptic case of Proposition \ref{cmax}.
 
 First, observe that  
 $$c \partial_x \underline{\psi} - d\partial_{yy}  \underline{\psi} = f(\underline{\psi}) + (c - \bar c)\partial_x  \underline{\psi} \leq f( \underline{\psi})$$
 Now, slide $\mu\phi, \psi$ to the left above $\mu \underline\phi, \underline \psi$ this way : just do it on a slice $x=a$ with $a > 0$ large enough so that on $x > a$ we know the sign of $\frac{f(\psi) - f(\underline \psi)}{\psi - \underline \psi}$ and can use the parabolic maximum principle with initial "time" $x=a$ (dotted line below) :
 
\begin{equation*}
  \begin{tikzpicture}
  \draw (-7,0) -- (7,0) node[pos=0.5,below] {\small{$d\partial_y (\psi - \underline\psi) + (\psi - \underline\psi) = \mu(\phi -\underline\phi)$}} node[pos=0.5,above] {$-(\phi -\underline\phi)'' + c(\phi -\underline\phi)' =(\psi - \underline\psi) - \mu (\phi -\underline\phi)$}; ;
  
    \node at (6,0.33) {$\phi -\underline\phi\to 0$};
    \node at (-6,0.33) {$0 \leftarrow \phi -\underline\phi$};

  \node at (0,-1.5) {$c\partial_x (\psi - \underline\psi) - d\partial_{yy} (\psi - \underline\psi) - \frac{f(\psi) - f(\underline \psi)}{\psi - \underline \psi}(\psi - \underline\psi) \geq 0$};

  \draw (-7,-3) -- (7,-3) node[pos=0.5,above] {\small{$\partial_y (\psi - \underline\psi)= 0$}};
  
  \draw[dashed] (-4.6,-3) -- (-4.6,0) node[pos=-0.05,below] {$\mu\phi,\psi > \mu\underline{\phi}, \underline \psi > 1 - \varepsilon$};


  \node at (6,-1.5) {$\psi - \underline\psi \to 0$};
  \node at (-6,-1.5) {$0 \leftarrow \psi - \underline\psi$};
  \end{tikzpicture}
\end{equation*}
 Treating the upper boundary as before, we obtain that $\mu\phi,\psi > \mu\underline{\phi}, \underline \psi$ on $x \geq a$. Using Proposition \ref{ordre}, the order is also true for $x$ negative enough, so that there is only a compact rectangle left where the order is needed : for this, just slide $\mu\phi, \psi$ enough to the left.

Now, as before, slide back to the right until the order is not true any more, finishing with the minimum possible translate $\mu\phi(r_0 +x) \geq \mu \underline\phi(x), \psi (r_0 + x,y) \geq \underline\psi(x,y)$. The strong parabolic maximum principle (without sign assumption) gives that the order is still strict (use a starting $x$ smaller than the $x$ where an eventual contact point happens) since $\underline \psi \not\equiv \psi$. Thus, on any compact $K_a$ as large as we want, we can slide $\mu \phi,\psi$ $\varepsilon_a$ more to the right again, the order still being true on $K_a$. Now just chose $a$ large enough so that $-a <  X +r_0 - \varepsilon_a$, and so that on $x > a$ we know the sign of $\frac{f(\psi) - f(\underline \psi)}{\psi - \underline \psi}$ : Prop. \ref{ordre} and the parabolic strong maximum principle give that  $\mu\phi(r_0 - \varepsilon_a +x) > \mu \underline\phi(x), \psi (r_0 - \varepsilon_a + x,y) > \underline\psi(x,y)$ which is a contradiction.

\end{proof}

\begin{rmq}\
\begin{itemize}
\item We could avoid the use of exponential solutions in the proof of Lemma \ref{ordre} : indeed, by considering some fixed translates $\phi_n^r, \psi_n^r$ of $\phi, \psi$ we can have, for $n$ large enough thanks to the locally uniform convergence and if $r$ is large enough

$\psi_n^r(-a,y) \geq (1-\delta)\psi^r(-a,y) \geq (1+\delta) \underline \psi(-a,y) \geq \underline \psi_n(-a,y)$

\item This idea of using the parabolic maximum principle to treat a degenerate elliptic equation motivates the following section where we answer the question : can the solution of \eqref{parab} be recovered by a direct method, without seeing it as the limit of the more regular solutions of \eqref{rescaled} ?
\end{itemize}
\end{rmq}

\section{Direct study of the limiting problem}
\label{direct}
We investigate the following elliptic-parabolic non-linear system in $$[0,M] \times([0,M] \times[-L,0]) = [0,M]\times \Omega_{L,M}$$
\begin{equation}
\label{boxed}
\begin{tikzpicture}
\draw (-5,0) -- (-5,-3) node[pos=0.5,left] {$v  = 0$};

\draw (-5,0) -- (4,0) node[pos=0.5,below] {\small{$d\partial_y v + v= \mu u$}} node[pos=0.5,above] {$-u'' + cu' + \mu u - v = 0$} node[pos=-0.05,above] {$u=0$} node[pos=1,above] {$u=1/\mu$} ;
\ptn(-5,0)(1.5);\ptn(4,0)(1.5);

\node at (0,-1.5) {$c  \partial_x v - d\partial_{yy} v = f(v)$};

\draw (-5,-3) -- (4,-3) node[pos=0.5,above] {$-\partial_y v  = 0$};

\draw[dashed] (4,0) -- (4,-3) node[pos=0.5,right] {$v  = ?$};

\end{tikzpicture}
\end{equation}
with $c,u(x),v(x,y)$ as unknowns.
We call a supersolution of \eqref{boxed} a solution of \eqref{boxed} where the $=$ signs are replaced by $\geq$. The plan of this section is the following : \begin{itemize}
\item First, we study the linear background of \eqref{boxed} in order to use Perron's method.
\item Then we prove the well-posedness of \eqref{boxed} and study monotonicity and uniqueness properties of the solution.
\item In a third subsection, we study the influence of $c$.
\item Finally, under a suitable normalisation condition on $c_M$ obtained thanks to the previous step, we study the limit $M \to +\infty$ of \eqref{boxed} and recover the solution of \eqref{parab}.
\end{itemize}

\subsection{Linear background}
In this subsection, we recreate the standard tools behind Perron's method. Even though these are quite standard, we give the proofs in our precise case because of the specificity of mixing parabolic and elliptic theory. Let $k > 0$ be a constant. We look at the following linear system of inequations
\begin{alignat*}{1}
&\begin{cases}
  c\partial_x v - d\partial_{yy}v + kv = h \text{ in } \Omega_{L,M} \\
  d\partial_y v(\cdot,0) + v(\cdot,0) \geq \mu u \text{ on } (0,M) \\
  -\partial_y v \leq 0 \text{ on } y=-L\\
  -\partial_{xx} u + c \partial_x u + (\mu+k)u - v(\cdot,0) = g \text{ in } (0,M) \\
\end{cases}
\end{alignat*}
along with the parabolic and elliptic limiting conditions
$$v \geq 0 \text{ on } x=0$$
$$ u(0) \geq 0, u(M) \geq 0$$
represented from now on as the following diagram
\begin{equation}
\label{lin}
\begin{tikzpicture}
\draw (-5,0) -- (-5,-3) node[pos=0.5,left] {$v \geq 0$};

\draw (-5,0) -- (4,0) node[pos=0.5,below] {\small{$d\partial_y v + v \geq \mu u$}} node[pos=0.5,above] {$-u'' + cu' + (\mu+k)u - v = g$} node[pos=-0.05,above] {$u \geq 0$} node[pos=1,above] {$u \geq 0$} ;

\node at (0,-1.5) {$c  \partial_x v - d\partial_{yy} v + kv = h$};

\draw (-5,-3) -- (4,-3) node[pos=0.5,above] {$-\partial_y v  \geq 0$};

\draw[dashed] (4,0) -- (4,-3) node[pos=0.5,right] {$v  = ?$};
\ptn(-5,0)(1.5);\ptn(4,0)(1.5);

\end{tikzpicture}
\end{equation}

\begin{prop}{(Maximum principle.)}
\label{maxp}
If $(u,v)$ are $\mathcal C^2$ functions up to the boundary of resp. $(0,M)$ and $\Omega_{L,M}$ that solve inequation $\eqref{lin}$ with $g,h \geq 0$ then $$u,v \geq 0$$ Moreover, $u,v > 0$ in resp. $(0,M)$ and $\Omega_{L,M}$ or $u \equiv 0, v\equiv 0$.
\end{prop}
\begin{proof}
We simply mix parabolic and elliptic strong maximum principles. Suppose that $\min v < 0$. By the parabolic strong maximum principle and Hopf's lemma, $\min v$ is necessarily reached on $y=0$ and at this point, $\mu u < \min v$. This is a contradiction with $$-\partial_{xx}u + c\partial_x u + (\mu + k) u \geq v$$ with its endpoints conditions that ensure $$u \geq \frac{\min v}{\mu+k}$$ Thus we have $v\geq 0$ and the elliptic maximum principle gives also $u \geq 0$. Finally, if $v(x_0,y_0) = 0$ with $x_0 > 0$ then by the strong parabolic maximum principle, $v \equiv 0$ on $x < x_0$, thus $u \equiv 0$ on $x < x_0$ which by the elliptic strong maximum principle gives $u  \equiv 0$, so that $v = g = u \equiv 0$ and $\partial_y v \geq 0$ on $y=0$, and by parabolic Hopf's lemma, $v \equiv 0$.
\end{proof}

\begin{cor}{(Comparison principle.)}
 Let $k > \text{Lip}(f)$. Then we have the following comparison principle :
 if $g_1 \leq g_2$ and $h_1 \leq h_2$, then if $(u_1,v_1)$ and $(u_2,v_2)$ are solutions $\mathcal C^2$ up to the boundary of 
\begin{equation}
\label{bizarre}
\begin{tikzpicture}
\draw (-5,0) -- (-5,-3) node[pos=0.5,left] {$v_i  = 0$};

\draw (-5,0) -- (4,0) node[pos=0.5,below] {\small{$d\partial_y v_i + v_i= \mu u_i$}} node[pos=0.5,above] {$-u_i'' + cu_i' + (\mu+k)u_i - v_i = kg_i$} node[pos=-0.05,above] {$u_i=0$} node[pos=1,above] {$u_i=1/\mu$} ;

\node at (0,-1.5) {$c  \partial_x v_i - d\partial_{yy} v_i + kv_i = f(h_i) + kh_i$};

\draw (-5,-3) -- (4,-3) node[pos=0.5,above] {$-\partial_y v_i  = 0$};
\draw[dashed] (4,0) -- (4,-3) node[pos=0.5,right] {$v_i  = ?$};
\ptn(-5,0)(1.5);\ptn(4,0)(1.5);

\end{tikzpicture}
\end{equation}
then $$u_1 \leq u_2, v_1 \leq v_2$$

\end{cor}
\begin{proof}
 Just observe that $(u_2-u_1,v_2-v_1)$ solve \eqref{lin} with $g = k(g_2-g_1) \geq 0$ and $h = f(h_2) - f(h_1) + k(h_2-h_1) \geq (-Lip(f) + k)(h_2-h_1) \geq 0$.
\end{proof}

\begin{cor}{(Supersolution principle.)} Let $(\bar u, \bar v)$ be a supersolution of \eqref{boxed}. If $(u,v)$ is a solution of \eqref{bizarre} with data $(\bar u, \bar v)$ then $$u \leq \bar u, v\leq \bar v$$
\end{cor}
\begin{proof}
 Observe that $(\bar u - u, \bar v - v)$ solves an inequation \eqref{lin} with $g,h \geq 0$.
\end{proof}

\begin{prop}({Unique solvability of the linear system.})
 Let $c > 0$,  $(g,h) \in \mathcal C^\alpha(\mathbb[0,M])\times\mathcal C^{\alpha/2,\alpha}(\Omega_{L,M})$ and $k>\text{Lip} f$. Then there exists a unique solution $(u,v)\in \mathcal C^{2,\alpha}(\mathbb[0,M])\times\mathcal C^{1+\alpha/2,2+\alpha}(\Omega_{L,M})$ of 

\begin{equation}
\label{lingh}
\begin{tikzpicture}
\draw (-5,0) -- (-5,-3) node[pos=0.5,left] {$v  = 0$};

\draw (-5,0) -- (4,0) node[pos=0.5,below] {\small{$d\partial_y v + v= \mu u_i$}} node[pos=0.5,above] {$-u'' + cu' + (\mu+k)u - v = g$} node[pos=-0.05,above] {$u=0$} node[pos=1,above] {$u=1/\mu$} ;

\node at (0,-1.5) {$c  \partial_x v- d\partial_{yy} v + kv = h$};

\draw (-5,-3) -- (4,-3) node[pos=0.5,above] {$-\partial_y v  = 0$};
\draw[dashed] (4,0) -- (4,-3) node[pos=0.5,right] {$v  = ?$};
\ptn(-5,0)(1.5);\ptn(4,0)(1.5);

\end{tikzpicture}
\end{equation}
\end{prop}

\begin{proof}
 The classical parabolic theory allows us to set $$S : \mathcal C^{1+\alpha/2} \to \mathcal C^{1+\alpha/2}, \quad U \mapsto v(\cdot,0)$$ where $v$ solves the last four equations in \eqref{lingh} with $u$ replaced by $U$. $S$ is affine and thanks to parabolic Hopf's lemma, uniformly continuous for the $L^\infty$ norm : $$|SU_1-SU_2|_\infty \leq \mu|U_1-U_2|_\infty$$  Since $\mathcal C^{1+\alpha/2}([0,M])$ is dense in $BUC([0,M])$, we can extend $S$ to a uniformly continuous affine function $\tilde S$ on $X=BUC([0,M])$.

 On the other hand thanks to classical ODE theory we can set $$T : L^\infty \to W^{2,\infty}, \quad V \mapsto u$$ where $u$ is solution of the first equation in \eqref{lin} with $v(\cdot,0)$ replaced by $V$. Observe also thanks to elliptic regularity that $T$ sends $\mathcal C^\alpha$ to $\mathcal C^{2,\alpha}$.

 By the strong elliptic maximum principle, observe that $T \circ \tilde S : X \to X$ is a contraction mapping : $$|T\tilde SU_1-T\tilde SU_2|_\infty \leq \frac{\mu}{\mu+k}|U_1-U_2|_\infty$$ By use of Banach fixed point theorem, it has a unique fixed point $u \in X$. Observe now that $u = T(\tilde S(u))$ and since $\tilde S(u) \in L^\infty$, $u = T(\tilde S(u)) \in  W^{2,\infty} \subset \mathcal C^{1+\alpha/2}$ and in the end $\tilde S(u) = S(u) \in \mathcal C^{1+\alpha/2}$ so that $u=T(S(u)) \in \mathcal C^{2+\alpha}$. Finally, parabolic regularity gives $v\in \mathcal C^{1+\alpha/2,2+\alpha}$ and $(u,v)$ solves \eqref{lin} in the classical sense. 

\end{proof}

\subsection{The non-linear system}

Combining all the results from the previous section we get : 
\begin{thm}
 There exists a  smooth solution $0 \leq \mu u, v \leq 1$ of \eqref{boxed}.
\end{thm}
\begin{proof}
 Use $(0,0)$ and $(1/\mu,1)$ as sub and supersolutions and start an iteration scheme from $(1/\mu,1)$. We get a decreasing sequence bounded from below by $(0,0)$. It converges point wise but the $L^\infty$ bound on $u_n,v_n$ gives a $\mathcal C^{1+\alpha/2}$ bound on $u_n$ which then gives a $\mathcal C^{1+\alpha/2,2+\alpha}$ bound on $v_n$, which then gives a $\mathcal C^{2+\alpha}$ bound on $u_n$. By Ascoli's theorem we can extract from $(u_n,v_n)$ a subsequence that converges to $(u,v) \in \mathcal C^{2+\beta}\times \mathcal C^{1+\beta/2,2+\beta}$. The point wise limit then gives the uniqueness of this limit point and thus that $(u_n,v_n)$ converges to it. Finally, $(u,v)$ has to be a solution of the equation. 
 
 Observe also that the only possible loss of regularity comes from the non-linearity $f$. Actually, if $f$ is of class $\mathcal C^\infty$, by elliptic and parabolic regularity described above, $u$ and $v$ are $\mathcal C^\infty$ too. More precisely, $f \in W^{k+1,\infty}$ implies $(u,v)\in C^{2+k,\alpha}(\mathbb[0,M])\times\mathcal C^{k+1+\alpha/2,k+2+\alpha}(\Omega_{L,M})$.
\end{proof}

We are now interested in sending $M \to +\infty$ to recover the travelling wave observed in the last section. For this, we need to normalise the solution in $\Omega_{L,M}$ in such a way that we do not end up with the equilibrium $(0,0)$ or $(1/\mu,1)$. We trade this with the freedom to chose $c$ : this motivates the investigation of the influence of $c$ on $(u,v)$ as well as a priori properties of $(u,v)$. To this end, we use a sliding method in finite cylinders. Because we apply the parabolic maximum principle on $v$, we will not have to deal with the corners of the rectangle.

\begin{prop}
 If $c > 0$ and $0 \leq \mu u, v \leq 1$ is a classical solutions of \eqref{boxed} then $$u', \partial_x v > 0$$ 
\end{prop}

\begin{proof}
 First observe that $m_0 := \min_{[-L,0]} v(M,y) > 0$ for the same reasons as in Prop. \ref{maxp}. Observe also that $$\lim_{\varepsilon \to 0}\max_{[0,\varepsilon]\times[-L,0]} v = 0$$ thanks to the uniform continuity of $v$. Denote $$v_r(x,y) = v(x+r,y)$$ and $$\Omega_{L,M}^r = [-r,M-r]\times[-L,0]$$ The previous observation asserts that $v_r - v > 0$ on $\Omega_{L,M}^r \cap \Omega_{L,M}$ if $r$ is close enough to $M$. Call $(r_0,M)$ a maximal interval such that for all $r$ in this interval, $v_r - v > 0$ on $\Omega_{L,M}^r \cap \Omega_{L,M}$. We know that such an interval exists by the previous observation. Let us show that $r_0 = 0$ by contradiction. Suppose $r_0 > 0$. By continuity, $v_{r_0} - v \geq 0$. But $(v_{r_0} - v)(0,y) > 0$ and $V: = v_{r_0} - v, U: = u_{r_0} - u$ satisfy on $[0,M-r_0], [0,M-r_0]\times[-L,0]$~:

\begin{equation*}
\begin{tikzpicture}
\draw (-5,0) -- (-5,-3) node[pos=0.5,left] {$V  > 0$};

\draw (-5,0) -- (4,0) node[pos=0.5,below] {\small{$d\partial_y V + V= \mu U$}} node[pos=0.5,above] {$-U'' + cU' + \mu U  = V$} node[pos=-0.05,above] {$U > 0$} node[pos=0.92,above] {$ U(M-r_0) > 0$} ;

\node at (0,-1.5) {$c  \partial_x V - d\partial_{yy} V + \frac{f(v_{r_0}) - f(v)}{v_{r_0} - v}V \geq 0$};

\draw (-5,-3) -- (4,-3) node[pos=0.5,above] {$-\partial_y V = 0$};
\draw[dashed] (4,0) -- (4,-3);
\draw[dashed] (3.2,0) -- (3.2,-3);
\end{tikzpicture}
\end{equation*}
By the mixed elliptic-parabolic strong maximum principle and Hopf's lemma for comparison with $0$ as in prop. \ref{maxp} we know that $v_{r_0} - v > 0$ (we cannot have $v_{r_0} \equiv v$ because then $u(M-r_0) = 1/\mu$ and that is impossible thanks to strong elliptic maximum principle since $r_0 > 0$). Then we may translate a little bit more, since $v_{r_0-\varepsilon} - v$ is continuous in $\varepsilon$, so that $v_{r_0-\varepsilon} - v > 0$, which is a contradiction with the definition of $r_0$.

As a result, $u$ and $v$ are non-decreasing in $x$, that is $u',\partial_x u \geq 0$. Now differentiating the equation with respect to $x$ and applying the same mixed maximum principle as above for comparison with $0$ yields $u', \partial_x v > 0$.
\end{proof}

\begin{prop}
\label{existence_hypo}
 For fixed $c > 0$, there is a unique solution $(u,v)$ of \eqref{boxed} such that $0\leq~\mu u, v \leq~1$. 
\end{prop}
\begin{proof}
 The proof is essentially the same as monotonicity. Suppose $(u_1,v_1)$ and $(u_2,v_2)$ are solutions. The same observations as above allow us to translate $v_2$ to the left above $v_1$. Now translate it back until this is not the case any more. We show by contradiction that this never happens : suppose $v_2^{r_0} - v \geq 0$ with $r_0 > 0$. Then $v_2^{r_0} - v > 0$ or $\equiv 0$, but the latter case is not possible since it would give $u(M-r_0) = \frac{1}{\mu}$. Thus $v_2^{r_0} - v > 0$ and we can still translate a bit more, that is a contradiction with the definition of $r_0$ so $r_0 = 0$ and $v_2 \geq v$. By symmetry, $v_2 \equiv v_1$, and then $u_2 \equiv u_1$.
\end{proof}

\begin{prop}
 The function $c \mapsto (u,v)$ is decreasing in the sense that if $0 < c < \bar c$ and $(c,u,v)$ and $(\bar c,\underline u, \underline v)$ solve \eqref{boxed} then $$\underline u < u, \underline v < v$$
\end{prop}

\begin{proof}
 The proof is again the same, just observe that $(\underline u, \underline v)$ is a subsolution of the equation with $c$, thanks to monotonicity since $(c-\bar c)\underline \partial_x v < 0$.
\end{proof}

\subsection{Limits with respect to $c$}
For now, we assert the following properties, which will be enough to conclude this section. Nonetheless, note that the study of solutions of \eqref{boxed} with $c=0$ shows interesting properties. Namely, the solutions are necessarily discontinuous, which implies that the regularisation comes also from the $c\partial_x$ term. This has to be seen in the light of hypoelliptic regularisation in kinetic equations (see \cite{Bou,Hor}). This study will be done in \cite{Mathese}.

\begin{prop}
\label{convchoc}
 Let $(u_c,v_c)$ denote the solution in Prop. \ref{existence_hypo}. Then for every $\alpha \in (0,1)$ $$\lim_{c\to 0} u_c(\alpha M) \geq \alpha$$
\end{prop}

\begin{proof}
 By monotonicity we already know that $(u_c,v_c)$ converges point wise as $c\to 0$ to some $(u,v)$. Since $0 \leq v_c \leq 1$, by uniqueness of the limit and Ascoli's theorem we know that $u$ is in $W^{2,\infty} \hookrightarrow \mathcal C^{1,1/2}$.

 Multiplying the equation satisfied by $v_c$ by a test function $\phi(y) \in \mathcal C^\infty_c(-L,0)$, integrating, and then multiplying by $\psi(x) \in \mathcal C^\infty_c(0,M)$ and integrating again, yields after passing to the limit $c\to 0$ thanks to dominated convergence : for all $x\in (0,M)$
$$ -d\int_{-L}^0 v\phi''(y) \mathrm dy = \int_{-L}^0 f(v(x,y)) \phi(y) \mathrm dy $$
i.e. $v(x,\cdot)$ satisfies $-v(x,\cdot)'' = f(v) \in L^\infty$ in the sense of distributions, so that $v(x,\cdot) \in W^{2,\infty} \hookrightarrow \mathcal C^{1,1/2}$, so that actually $f(v) \in \mathcal C^1$ and these equations are satisfied in the classical sense. Moreover $v(x,\cdot)'$ is bounded and by concavity, $v_y(-L)$ exists and is seen to be necessarily $0$ by using a test function $\phi$ whose support intersects $y=-L$ and integrating by parts.

At this point, $v(x,\cdot) \in \mathcal C^2(-L,0)\cap \mathcal C^1([-L,0])$ and satisfies $-d\partial_{yy} v(x,y) = f(v(x,y))$ and $\partial_y v(-L)=0$. Using now a test function $\phi$ whose support intersects $y=0$ and integrating by parts we obtain $d \partial_y v(x,0)= \mu u(x) - v(x,0)$.

Moreover, $u(0) = v(0,y) = 0$, $u(M) = 1/\mu$, $v$ is non-decreasing with respect to the $x$ variable so differentiable a.e. and continuous up to a countable set (which in fact is of cardinal $1$, $2$ or $3$).

Finally, passing to the limit in the equation for $u$ yields that $u \in \mathcal C^1$ satisfies $-u'' = \mu u - v$ in the sense of distributions, so that we have the following picture :

\begin{equation}
\label{boxedlimit0}
\begin{tikzpicture}
\draw (-5,0) -- (-5,-3) node[pos=0.5,left] {$v  = 0$};

\draw (-5,0) -- (4,0) node[pos=0.5,below] {\small{$d\partial_y v = \mu u - v$}} node[pos=0.5,above] {$-u'' = v - \mu u$} node[pos=-0.05,above] {$u=0$} node[pos=1,above] {$u=1/\mu$} ;

\node at (0,-1.5) {$- d\partial_{yy} v = f(v)$};

\draw (-5,-3) -- (4,-3) node[pos=0.5,above] {$-\partial_y v  = 0$};

\draw[dashed] (4,0) -- (4,-3) node[pos=0.5,right] {$v  = ?$};
\ptn(-5,0)(1.5);\ptn(4,0)(1.5);

\end{tikzpicture}
\end{equation}
Notice that since $f \geq 0$, $v(x,\cdot)$ is concave, so that $d\partial_y v(x,0) \leq 0$, i.e. $v-\mu u \geq 0$ so that $u$ is concave. As a consequence, it is over the chord between $(0,0)$ and $(M,1/\mu)$ which is the desired conclusion.

\end{proof}

\begin{prop}
 $$\lim_{c\to +\infty} u_c = 0\text{ pointwise on }[0,M[$$
\end{prop}
\begin{proof}
 Applying the exact same method as above we obtain as $c\to +\infty$, $(u,v)$ a classical solution of 

\begin{equation}
\label{boxedlimitinf}
\begin{tikzpicture}
\draw (-5,0) -- (-5,-3) node[pos=0.5,left] {$v  = 0$};

\draw (-5,0) -- (4,0) node[pos=0.5,below] {\small{$d\partial_y v = \mu u - v$}} node[pos=0.5,above] {$u' = 0$} node[pos=-0.05,above] {$u=0$} node[pos=1,above] {$u=1/\mu$} ;

\node at (0,-1.5) {$ \partial_x v = 0$};

\draw (-5,-3) -- (4,-3) node[pos=0.5,above] {$-\partial_y v  = 0$};
\draw[dashed] (4,0) -- (4,-3) node[pos=0.5,right] {$v  = ?$};
\ptn(-5,0)(1.5);\ptn(4,0)(1.5);

\end{tikzpicture}
\end{equation}
Now just observe that $u_c'(0)$ is decreasing with respect to $c$ and non-negative, so that $u$ can be extended in a $\mathcal C^1$ way on $[0,M)$ by the Cauchy criterion : we end-up with $u \equiv 0$ on $[0,M)$. We observe that $u$ is thus necessarily discontinuous at $x=M$, which is consistent with the limit $c\to \infty$ in the integration by parts of the equation on $(u_c,v_c)$ :

 $$c\left(\frac{1}{\mu} + \int_{-L}^0 v_c(M,y) \mathrm dy\right) = \int_{\Omega_{L,M}} f(v_c) + u_c'(M) - u_c'(0)$$
since the left-hand side goes to $+\infty$ and everything except $u_c'(M)$ is bounded by above in the right-hand side.

\end{proof}

\subsection{Limit as $M\to\infty$}
We now call $\theta' = \frac{1+\theta}{2}$, $\theta'' = \frac{2+\theta}{3}$ and chose $c = c_M$ such that $u_{M}(\theta'' M) = \theta'$. We are now interested in compactness on $c_M$ to pass to the limit $M\to+\infty$ in the equations. From now on we change the coordinate $x$ by $x-\theta' M$ so that $u_M, v_M$ are resp. defined on $$\Omega_M := [-\theta'' M, (1-\theta'') M]$$ $$\Omega_{L,M} := [-\theta'' M, (1-\theta'') M]\times [-L,0]$$ and $$u_M(0) = \theta'$$

\begin{prop}
\label{cmaxM}
 For $M$ large enough, $c_M < \sqrt{\text{Lip} f}$
\end{prop}

\begin{proof}
 Since we do not know how the level lines $\{ v_M = \theta \}$ behave, we cannot apply the argument of Proposition \ref{cmax}. Nonetheless, this is counterbalanced by taking advantage of being in a rectangle. We look for 
$$\bar u(x) = e^{rx}, \bar v(x,y) = \mu e^{rx}$$ to solve \eqref{boxed} with the $=$ signs replaced by $\geq$. A direct computation yields $$-r^2 + c_M r \geq 0, c_M r \geq \text{Lip} f$$ so the best choice is $$r = c_M, c_M \geq \sqrt{\text{Lip} f}$$ So now suppose $$c_M \geq \sqrt{\text{Lip} f}$$ and set 
$$t_0 = \inf \{ t\in \mathbb R \mid \bar u(t+\cdot) - u_M) > 0\text{ on } \Omega_M \textbf{ and } \bar v(t + \cdot,y) - v_M > 0 \text{ on } \Omega_{L,M} \}$$
This infimum exists as it is taken over a set that is non-void and bounded by below (using the limits of $e^{rx}$ and the bounds on $u,v$). By continuity $$\bar u(t_0 + \cdot) - u_M, \bar v(t_0 + \cdot,y) - v_M \geq 0$$ and 

\begin{equation*}
 \begin{split}
\text{There exists } x_0 \in \Omega_M &\text{ s.t. }\bar u(t_0 + x_0) - u(x_0) = 0 \\ &\textbf{or} \\
\text{There exists }(x_0',y_0') \in \Omega_{L,M} &\text{ s.t. }\bar v(t_0 + x_0',y_0') - v(x_0',y_0') = 0
 \end{split}
\end{equation*}
Since $u = v = 0$ at the left boundaries, $x_0, x_0' > -\theta'' M$. Thanks to the normalisation condition, the first case is impossible, since $\bar u(t_0 + \cdot) - u_M$ satisfies the strong elliptic maximum principle with non-negative boundary values and data. Indeed, the only thing to check is that 
$$\bar \mu u(t_0+(1-\theta'') M) \geq \mu u_M((1-\theta'') M) = 1 $$
This is obtained provided $M \geq \frac{2}{\sqrt{\text{Lip} f}} \ln\left(\frac{1}{\mu}\right)$ : the level lines $\theta'$ should touch before the level lines $1$ since 

\begin{equation*}
\begin{split}
 | \mu u_M^{-1}(\theta') - \mu \bar u^{-1}(\theta') | &= \frac{1}{c} \ln\left(\frac{\theta'}{\mu}\right) \\
						     &< \frac{1}{c} \ln\left(\frac{1}{\mu}\right) \\
						     &< M - \frac{1}{c} \ln\left(\frac{1}{\mu}\right) \\
						     &= |\mu u_M^{-1}(1) - \mu \bar u^{-1}(1)|
\end{split}
\end{equation*}

The second case is impossible also, by using the strong parabolic maximum principle and Hopf's lemma as usual.
In every case, there is a contradiction so that in the end 

$$c_M \leq \sqrt{\text{Lip} f}$$

\end{proof}

\begin{prop}
 There exists $c_- > 0$ that does not depend on $M$ s.t. $c_M > c_-$
\end{prop}

\begin{proof}
We argue by contradiction by supposing $\inf c(M) = 0$. Then there exists $M_n \to \infty$ (by continuity of $c(M)$) such that $c_{M_n} =: c_n \to 0$. Denote $u_n,v_n$ the associated normalised sequence of solutions. Since $u_n$ is uniformly in $\mathcal C^{1,\beta}$ we extract from it a subsequence that converges in $\mathcal C^{1,\alpha}$. We now assert the following :

\begin{center}
For every $A>0$, $(v_n(y)(x))_n$ is equicontinuous and bounded in $\mathcal C( (-L,0), L^1(-A,A))$
\end{center}
The boundedness comes directly from the fact that $v_n(x,y) \in [0,1]$ is increasing, thus it is bounded uniformly in $n$ and $y$ in $BV(-A,A)$ which is compactly embedded in  $L^1(-A,A)$. For the equicontinuity, we have
$$\int_{-A}^A | v_n(y,x) - v_n(y+\varepsilon,x) | dx \leq \int_{-A}^A \left(\int_y^{y+\varepsilon} \left|\partial_y v_n(x,s) \right| ds \right) dx$$
so that a  uniform bound on $\partial_y v_n$ will suffice. This bound is classical and comes from parabolic regularity after rescaling, but let us give it here for the sake of completeness. Consider $\mathfrak u_n(x) = u(c_nx)$ and $\mathfrak v_n(x,y) = v(c_nx,y)$ so that with the new variables, $x\in (- \frac{\theta' M_n}{c_n},\frac{(1-\theta') M_n}{c_n})$ and $\mathfrak u, \mathfrak v$ satisfy

\begin{equation}
\label{rescaleparab}
\begin{tikzpicture}
\draw (-5,0) -- (-5,-3) node[pos=0.5,left] {$\mathfrak v  = 0$};

\draw (-5,0) -- (4,0) node[pos=0.5,below] {\small{$d\partial_y \mathfrak v = \mu \mathfrak u - \mathfrak v$}} node[pos=0.5,above] {$-\mathfrak u'' = \mathfrak v - \mu \mathfrak u$} node[pos=-0.05,above] {$ \mathfrak u=0$} node[pos=1,above] {$\mathfrak u=1/\mu$} ;

\node at (0,-1.5) {$ \partial_x \mathfrak v - d\partial_{yy} \mathfrak v = f(\mathfrak v)$};

\draw (-5,-3) -- (4,-3) node[pos=0.5,above] {$-\partial_y \mathfrak v  = 0$};
\draw[dashed] (4,0) -- (4,-3) node[pos=0.5,right] {$\mathfrak v  = ?$};
\ptn(-5,0)(1.5);\ptn(4,0)(1.5);

\end{tikzpicture}
\end{equation}
Now we reduce to a local estimate :
\begin{equation*}
\begin{split}
|\mu \mathfrak u_n|_{\mathcal C^{1,\alpha}(-A,A)} &=  |\mu \mathfrak u_n|_{L^\infty(-A,A)} +  |\mu \mathfrak u_n'|_{L^\infty(-A,A)} + \sup_{x\neq y} \frac{ |\mu \mathfrak u_n'(x)-\mu \mathfrak u_n'(y)| }{|x-y|^\alpha}  \\
&\leq |\mu \mathfrak u_n|_{L^\infty(-A,A)} + 3|\mu \mathfrak u_n'|_{L^\infty(-A,A)} + \sup_{|x-y| < 1} \frac{ |\mu \mathfrak u_n'(x)-\mu \mathfrak u_n'(y)| }{|x-y|^\alpha} \\
&\leq |\mu \mathfrak u_n|_{L^\infty(-A,A)} + 4\sup_{x_0\in(-A,A)} |\mu \mathfrak u_n|_{\mathcal C^{1,\alpha}(B_1(x_0))}
\end{split}
\end{equation*}
And finally, $|\mu \mathfrak u_n|_{L^\infty(-A,A)} \leq 1$ as well as 

\begin{equation*}
  \begin{split}
|\mu \mathfrak u_n|_{\mathcal C^{1,\alpha}(B_1(x_0))} &\leq C_0 |\mu \mathfrak u_n|_{W^{2,p}(B_1(x_0))} \\ 									 &\leq C_0C_1\left( |\mu \mathfrak u_n|_{L^p(B_1(x_0))} + |\mathfrak v_n|_{L^p(B_1(x_0))} \right) \\ 
							  &\leq 2C_0C_1
  \end{split}
\end{equation*}
thanks to the Sobolev inequality, the standard $W^{2,p}$ estimates with $p=1/(1-\alpha)$, and $0 \leq \mu u \leq 1$, so that in the end 
$$|\mu \mathfrak u_n|_{\mathcal C^{1,\alpha}(-A,A)} \leq 1 + 8C_0C_1$$
Finally, we plug this in the classical Schauder parabolic estimate up to the boundary to get 
$$|\mathfrak v_n|_{\mathcal C^{1+\alpha/2, 2+\alpha}(-A,A)} \leq C_3\left(|\mathfrak v_n|_{L^\infty(-A,A)} + |\mathfrak \mu u_n|_{\mathcal C^{1+\alpha/2}(-A,A)}\right) \leq C_3(2 + 8C_0C_1)$$
even independently from $A$.
So that in the end $$| \partial_y v_n |_{L^\infty((-A,A)\times(-L,0))} = | \partial_y \mathfrak v_n |_{L^\infty((-A/c_n,A/c_n)\times(-L,0))} \leq C_3(2 + 8C_0C_1) $$

The fact is now proved, and thanks to Ascoli's theorem and a diagonal extraction, we can extract from $u_n,v_n$ some $u,v$ that converges in $\mathcal C( (-L,0), L^1(-n,n) )$ for every $n \in \mathbb N$. Just as in the previous computations by integrating by parts, we get that $u,v$ ends up to be a classical solution of (since $M_n/c_n \to \infty$) 
\begin{equation}
\label{absurdeconcave}
  \begin{tikzpicture}
  \draw (-6,0) -- (6,0) node[pos=0.5,below] {\small{$d\partial_y v = \mu u - v(x,0)$}} node[pos=0.5,above] {$-u'' = v(x,0) - \mu u$};

  \node at (0,-1.5) {$- d\partial_{yy} v = f(v)$};

  \draw (-6,-3) -- (6,-3) node[pos=0.5,above] {\small{$-\partial_y v= 0$}};


  \end{tikzpicture}
\end{equation}
along with $\mu u(0) = (1+\theta)/2$. But this is impossible : indeed, $u$ is bounded and thanks to $f\geq 0$, $v$ is concave on each $y$-slice, which gives that $u$ is also concave, on the whole $\mathbb R$ so it is constant. Thanks to the normalisation condition, $\mu u \equiv (1+\theta)/2$, so that $v(x,0) \equiv \mu u \equiv (1+\theta)/2$, so that $\partial_y v(x,0) \equiv 0$, but then by concavity and the Neumann condition, $v \equiv (1+\theta)/2$ which is a contradiction with $f((1+\theta)/2) > 0$.

\end{proof}

We can now pass to the limit $M\to\infty$ in the equations and prove Theorem \ref{degen}. 

\begin{proof}
 Taking $M_n \to +\infty$, thanks to the bounds on $c_{M_n}$ we can extract from it a subsequence converging to some $c > 0$. We can also use the elliptic-parabolic regularity discussed in the beginning to extract from $(u_{M_n},v_{M_n})$ some subsequence that converges in $\mathcal C^{1+\alpha/2,2+\alpha}_{loc}$ to some $(u,v)$ that solve the equations in \eqref{parab}. Bounds and monotonicity are inherited from the $\mathcal C^1$ limit. The last thing to check are the uniform limits as $x \to \mp\infty$, which are obtained thanks to the following lemmas.
\end{proof}

\begin{lem}
\label{energy}
 $$\iint\limits_{[0,+\infty]\times[-L,0]} f(v) < +\infty, \iint\limits_{[0,+\infty]\times[-L,0]} | \partial_y v |^2 < +\infty$$ and the same is true with $[-\infty,0]$.
\end{lem}

\begin{proof}
Integrate on $[0,M]\times[-L,0]$ the equation for $v$ in \eqref{parab} to get
$$ \iint\limits_{[0,M]\times[-L,0]} f(v) = c\left( \int_{-L}^0 v(M,y) - \int_{-L}^0 v(0,y) \right) + u'(M) - u'(0) + c\left(u(M) - u(0)\right)$$
Everything in the left-hand side is bounded, apart from $u'(M)$. Thus, for the integral to diverge as $M\to+\infty$, $u'(M)  \to +\infty$, which is impossible since $u$ is bounded.

For the second integral, multiply the equation by $v$ and integrate by parts to get
$$ \iint\limits_{[0,M]\times[-L,0]} d|\partial_y v|^2 = \iint\limits_{[0,M]\times[-L,0]} f(v)v - c\int_{-L}^0 \frac{1}{2}(v(M,y)^2 - v(0,y)^2) + \int_0^M (u''- cu')v$$
The first two integrals in the right-hand side are bounded. For the last one, we see that $$c\int_0^M u'v \leq c(u(M)-u(0)) \leq \frac{c(1-\theta')}{\mu}$$ so that for the integral to diverge as $M\to+\infty$, $\int_0^M u''v \to +\infty$. But 
$$\int_0^M u'' v = u'(M)v(M) - u'(0)v(0) - \int_0^M u'v' \leq u'(M)v(M) \leq u'(M)$$ so that this is again impossible. The case of $[-\infty,0]$ is similar.

\end{proof}

\begin{prop}
\label{rightlimits}
 $\mu u(x), v(x,y) \to 1$ uniformly in $y$ as $x\to +\infty$ and to some constant $v_- \leq \theta$ as $x\to -\infty$.
\end{prop}

\begin{proof}
 By bounds and monotonicity, $v(x,y)$ converges point wise to some $v_+(y)$ as $x\to+\infty$. Let us define the functions $v_j(x,y) = v(x+j,y)$ in $[0,1]\times[-L,0]$ for every integer $j$. Standard parabolic estimates and Ascoli's theorem tell us that up to extraction, $v_j \to \delta$ in the $\mathcal C^1$ sense for a $\mathcal C^1$ function $\delta$. By uniqueness of the simple limit, $\beta = v_+ \in \mathcal C^1$. So $v_j$ lies in a compact set of $\mathcal C^1( [0,1] \times [-L,0])$ and has a unique limit point $\beta\in \mathcal C^1$ : then it converges to it in the $\mathcal C^1$ topology. The $y$-uniform limits follow.

 Now using the finiteness of the second integral above, we have that $v_+(y)$ is constant, moreover, $f(v_+) = 0$ thanks to the finiteness of the first integral. By the exchange condition, $\mu u$ converges to $v_+$ as $x\to +\infty$ so necessarily $v_+ \geq \theta'$ : the only possibility is $v_+ = 1$. 

 The exact same arguments apply to $-\infty$, but all of $[0,\theta]$ are admissible constants. Let us finish with the following.
\end{proof}

\begin{prop}
 $$v_- = 0$$
\end{prop}

\begin{proof}
 Here we use that $v$ comes from $v_{M_n}$. As in the proof of the upper bound on $c_M$, we do not know what happens to the level lines $\{ v_{M_n} = \theta \}$, which prevents us to use the usual comparison with positive exponential solutions as $v \leq \theta$ : indeed, the sets $\{ v_{M_n} = \theta \}$ could be sent to $-\infty$. We use a sliding method with a less sharp supersolution by looking at a level line $\{ v_{M_n} = \alpha \}$ with $\alpha > \theta$ close to $\theta$ to prove that this is not the case. We give below a picture of the argument before writing it completely.
 
 \begin{figure}[htb]
 \label{f}
  \centering 
  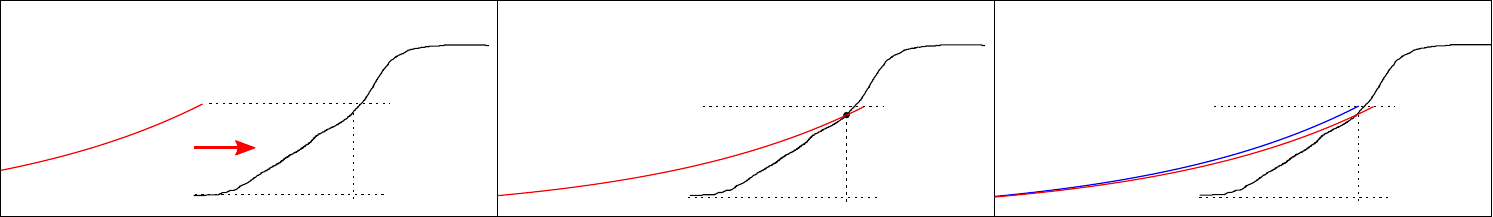 
 \end{figure}
First, observe that thanks to the $y$-uniform convergence of $\mu u,v$ to $v_- \leq \theta$ resp. $1$ as $x\to\mp\infty$ :
$$\exists \ x_- \leq x_+ \in \mathbb R\text{ s.t. } \\ \{\mu u_{M_n} = \alpha\}\subset[x_-,x_+] \text{ and } \{v_{M_n} = \alpha\} \subset [x_-,x_+]\times[-L,0]$$
Indeed, there exists $x_-, x_+$ s.t. for all $x\leq x_-, y\in[-L,0]$, $\mu u, v(x,y) \leq \frac{\theta + \alpha}{2}$ and for all $x\geq x_+, y\in[-L,0]$, $\mu u, v(x,y) \geq \frac{\alpha + 1}{2}$. Thanks to the uniform local convergence of $u_{M_n}, v_{M_n}$ to $u,v$ we can say that there exists $N\in \mathbb N$ s.t. for all $n \geq N$, on $[x_-,x_+]\times[-L,0]$ we have $\frac{2\theta + \alpha}{3} \leq \mu u_{M_n}, v_{M_n} \leq \frac{\alpha + 2}{3}$, and we conclude thanks to the monotonicity of $u_{M_n}, v_{M_n}$. 

We now work on $x \leq x_-$ so that $\mu u_{M_n}, v_{M_n} \leq \alpha$ and use the fact that $\text{Lip} f_{|[0,\alpha]}$ is as small as we want by taking $\alpha$ close enough to $\theta$. The same computations as in Prop. \ref{cmaxM} as well as $c > c_-$ and the monotonicity of $e^{c_-x}$ give that if $\alpha$ is chosen so that
$$\text{Lip} f_{|[0,\alpha]} \leq c_-^2$$
then $(e^{c_- x}, \mu e^{c_- x})$ is a supersolution of \eqref{boxed} as long as $\mu e^{c_- x} \leq \alpha$ : so look at the graph of $\mu e^{c_- x}$ and cut it after it reaches $\alpha$. Now translate this half-graph to the left until it is disconnected with the graph of $u_{M_n}, v_{M_n}$ and bring it back until it touches $\mu u_{M_n}$ or $v_{M_n}$ before $x_-$, which necessarily happens since $\mu u(x_-), v(x_-,y) \leq \alpha$. The arguments given in Prop. \ref{cmaxM} assert here that the contact necessarily happens at $x_-$ with $\mu u_{M_n}$, i.e. the graphs of $\mu u_{M_n}, v_{M_n}$ are below some translation of the cut graph of $\mu e^{c_- x}$ that touches it at $x_-$, where $\mu u_{M_n} \leq \alpha$ so that they are below the graph of $\alpha e^{c_-(x-x_-)}$, i.e.
$$\mu u_{M_n}, v_{M_n} \leq e^{c_-(x-x_-)}\text{ on }x \leq x_-$$
By making $n \to \infty$ we get that $\mu u, v$ decays as $x\to -\infty$ at least as $e^{c_- x}$, which is consistent with the computations of the exponential solutions in Section \ref{upboundDinf}. 
\end{proof}

As a conclusion, I would like to mention that this study motivates the question of convergence towards travelling waves. This work suggests that the travelling wave of \eqref{normal} is globally stable among initial data that are over $\theta$ on a set large enough (whose measure would scale as $\sqrt{D}$). I also conjecture that this convergence happens uniformly in $D$. This will be the purpose of a future work. 

\section*{Acknowledgement}
The research leading to these results has received funding from the European Research Council under the European Union’s Seventh Framework Programme (FP/2007-2013) / ERC Grant Agreement n.321186 - ReaDi -Reaction-Diffusion Equations, Propagation and Modelling.

\bibliographystyle{abbrv}
\bibliography{refs}

\end{document}